%% file: main.tex
\documentclass[journal,twoside,web]{ieeecolor}  % Comment this line out
\usepackage{lcsys}
\usepackage{url}
\usepackage{breakurl}
\usepackage{additionalSetting}
\usepackage{anyfontsize}

\definecolor{darkgreen}{rgb}{0.0, 0.5, 0.0}

\usepackage[normalem]{ulem}

\def\BibTeX{{\rm B\kern-.05em{\sc i\kern-.025em b}\kern-.08em
		T\kern-.1667em\lower.7ex\hbox{E}\kern-.125emX}}
\title{\LARGE \bf
Distributionally robust LMI synthesis for LTI systems
}

\author{Dennis Gramlich, Shuhao Yan, Carsten W. Scherer and Christian Ebenbauer% <-this % stops a space
\thanks{Funded by Deutsche Forschungsgemeinschaft (DFG, German Research Foundation) under Germany's Excellence Strategy - EXC 2075 – 390740016. We acknowledge the support by the Stuttgart Center for Simulation Science (SimTech).}
\thanks{Dennis Gramlich and Christian Ebenbauer are with the Chair of Intelligent Control Systems,
        RWTH Aachen University,
        52074 Aachen, Germany
        {\tt\small \{dennis.gramlich, christian.ebenbauer\} @ic.rwth-aachen.de}}%
\thanks{Carsten W. Scherer and Shuhao Yan are with the Chair of Mathematical Systems Theory, University of Stuttgart, 70569 Stuttgart, Germany {\tt\small \{carsten.scherer, shuhao.yan\}@imng.uni-stuttgart.de}}%
}%

\begin{document}

\maketitle
\thispagestyle{empty}
\pagestyle{empty}

%%%%%%%%%%%%%%%%%%%%%%%%%%%%%%%%%%%%%%%%%%%%%%%%%%%%%%%%%%%%%%%%%%%%%%%%%%%%%%%%
\begin{abstract}

	This article shows that distributionally robust controller synthesis as investigated in \cite{taskesen2024distributionally} can be formulated as a convex linear matrix inequality (LMI) synthesis problem. To this end, we rely on well-established convexification techniques from robust control. The LMI synthesis problem we propose has the advantage that it can be solved efficiently using off-the-shelf semi-definite programming (SDP) solvers. In addition, our formulation exposes the studied distributionally robust controller synthesis problem as an instance of robust $H_2$ synthesis.

\end{abstract}

\begin{IEEEkeywords}
	distributionally robust control, linear matrix inequalities, robust-$H_2$ control
\end{IEEEkeywords}

%%%%%%%%%%%%%%%%%%%%%%%%%%%%%%%%%%%%%%%%%%%%%%%%%%%%%%%%%%%%%%%%%%%%%%%%%%%%%%%%

\section{Introduction and problem statement}
\label{sec:1}

\IEEEPARstart{D}{istributionally} robust control (DRC) addresses control problems where the probability distributions governing system disturbances are not precisely known. Unlike classical robust control, which considers worst-case deterministic uncertainties, DRC optimizes performance against the worst-case probability distribution within an ambiguity set, making it particularly suited for data-driven scenarios where distributional information can be estimated from samples. The goal of this article is to showcase the fundamental relations between DRC, robust $H_2$ control, and moment relaxations for linear quadratic control problems. As our main result, we show that linear quadratic DRC problems are specific instances of robust $H_2$ control problems. Thus, efficient LMI-synthesis methods are applicable for these problems. At the same time, this result strengthens the foundation of classical robust control theory providing yet another problem that can be embedded into this framework.

Concretely, we study a (generalized) plant of the form
\begin{align}
  \begin{pmatrix}
			x(t+1)\\
			z(t)\\
			y(t)
		\end{pmatrix}
		=
		\begin{pmatrix}
			A & B_w & B_u\\
			C_z & D_{zw} & D_{zu}\\
			C_y & D_{yw} & 0
		\end{pmatrix}
		\begin{pmatrix}
			x(t)\\
			w(t)\\
			u(t)
		\end{pmatrix}
    \label{eq:generalizedPlant}
\end{align}
that is in feedback with a controller given by
\begin{align}
	\begin{pmatrix}
		x_c(t+1)\\
		u(t)
	\end{pmatrix}
	=
	\begin{pmatrix}
		\dvrC{A_c} & \dvrC{B_c}\\
		\dvrC{C_c} & \dvrC{D_c}
	\end{pmatrix}
	\begin{pmatrix}
		x_c(t)\\
		y(t)
	\end{pmatrix}.
  \label{eq:controller}
\end{align}
Here, the state $x$, the input $u$, the disturbance $w$, the performance output $z$, the measured output $y$, and the controller state $x_c$ are stochastic processes defined over a joint (separable, metric) sample space $\Omega$. Further, we assume $x(t)\in L_2(\Omega,\bbR^{n_x})$, $u(t) \in L_2(\Omega, \bbR^{n_u})$, $w(t) \in L_2(\Omega,\bbR^{n_w})$, $z(t) \in L_2(\Omega,\bbR^{n_z})$, $y(t) \in L_2(\Omega,\bbR^{n_y})$, and $x_c(t) \in L_2(\Omega,\bbR^{n_x})$ at all times $t \in \bbN_0$.

The matrices $A \in \bbR^{n_x \times n_x}$, $B_w \in \bbR^{n_x \times n_w}$, $B_u \in \bbR^{n_x \times n_u}$, $C_z \in \bbR^{n_z \times n_x}$, $D_{zw} \in \bbR^{n_z \times n_w}$, $D_{zu} \in \bbR^{n_z\times n_u}$, $C_y \in \bbR^{n_y \times n_x}$, $D_{yw} \in \bbR^{n_y \times n_w}$ are the system matrices. The system matrices are assumed to be known problem parameters. The matrices $\dvrC{A_c} \in \bbR^{n_x \times n_x}$, $\dvrC{B_c} \in \bbR^{n_x \times n_y}$, $\dvrC{C_c} \in \bbR^{n_u \times n_x}$, and $\dvrC{D_c} \in \bbR^{n_u \times n_y}$ define the controller. The controller matrices are colored in blue, since they are the decision variables in the optimization problem \eqref{eq:mainProblem}. Auxiliary optimization variables are colored in cyan (magenta), if we minimize (maximize) over them.

We aim to design a controller \eqref{eq:controller} for the plant \eqref{eq:generalizedPlant} by solving the control problem
\begin{align}
	\minimize_{\dvrC{A_c},\dvrC{B_c},\dvrC{C_c},\dvrC{D_c}} ~&~ v(\dvrC{A_c},\dvrC{B_c},\dvrC{C_c},\dvrC{D_c}), \label{eq:mainProblem}\\
  \mathrm{s.t.} ~&~
  \dvrC{\calA} \text{ is stable,} \nonumber
\end{align}
where the cost function is defined as
\begin{subequations}
	\label{eq:finiteHorizonProblem}
	\begin{align}
		v(\dvrC{A_c},\dvrC{B_c},\dvrC{C_c},\dvrC{D_c}) =
    \sup_{w \in \scrW} ~&~ \limsup_{T \to \infty} \frac{1}{T} \sum_{t = 0}^{T-1} \bbE \left\| z(t) \right\|^2 \label{eq:finiteHorizonProblemA}\\
		\mathrm{s.t.} ~&~ \begin{pmatrix}
    \chi(t+1)\\
    z(t)
  \end{pmatrix}
  =
  \begin{pmatrix}
    \dvrC{\calA} & \dvrC{\calB}\\
    \dvrC{\calC} & \dvrC{\calD}
  \end{pmatrix}
  \begin{pmatrix}
    \chi(t)\\
    w(t)
  \end{pmatrix}, \label{eq:finiteHorizonProblemC}\\
		&~ \chi(0) = 0. \label{eq:finiteHorizonProblemD}
	\end{align}
\end{subequations}

Here, $\chi^\top(t) = \begin{pmatrix} x^\top(t)& x_c^\top(t) \end{pmatrix}$ is the combined state of the closed-loop system between \eqref{eq:generalizedPlant} and \eqref{eq:controller}. The closed-loop matrices $\dvrC{\calA},\dvrC{\calB},\dvrC{\calC},\dvrC{\calD}$ are given by
\begin{align*}
  \setlength{\arraycolsep}{2pt}
  % \resizebox{\linewidth}{!}{$
	\left(
	\begin{array}{c|c}
		\dvrC{\calA} & \dvrC{\calB}\\
		\hline
		\dvrC{\calC} & \dvrC{\calD}
	\end{array}
	\right)
	:=
	\left(
	\begin{array}{cc|c}
		A + B_u \dvrC{D_c} C_y & B_u \dvrC{C_c} & B_w + B_u \dvrC{D_c} D_{yw}\\
		\dvrC{B_c} C_y & \dvrC{A_c} & \dvrC{B_c} D_{yw}\\
		\hline
		C_z + D_{zu} \dvrC{D_c} C_y & D_{zu} \dvrC{C_c} & D_{zw} + D_{zu} \dvrC{D_c} D_{yw}
	\end{array}
	\right).%$}
\end{align*}
Like the controller matrices, the closed-loop matrices are colored in blue, since they are affine functions of the controller parameters.

The ambiguity set $\scrW$ confines the probability distributions of the disturbance $w$ using the 2-Wasserstein metric. It may take one of the following forms:
\begin{enumerate}
	\item The ambiguity set $\scrW_\text{cor}$ contains all sequences of random disturbances $w: \bbN_0 \to L_2(\Omega,\bbR^{n_w})$ that satisfy $\bbW(\bbP_{w(t)},\bbP_\text{nom}) \leq \gamma$ for all $t \in \bbN_0$ with a fixed, zero mean, normal distribution $\bbP_\text{nom}$ with $\Sigma_\text{nom} := \int vv^\top \diff \bbP_\text{nom}(v)$.
	\item The ambiguity set $\scrW_\text{ind}$ contains all sequences of \emph{independent, zero mean} random disturbances $w: \bbN_0 \to L_2(\Omega,\bbR^{n_w})$ that satisfy $\bbW(\bbP_{w(t)},\bbP_\text{nom}) \leq \gamma$ for all $t \in \bbN_0$ for a fixed, zero mean, normal distribution $\bbP_\text{nom}$ with $\Sigma_\text{nom} := \int vv^\top \diff \bbP_\text{nom}(v)$.
\end{enumerate}

\section{Related work}
\label{sec:relatedWork}

\emph{Distributionally robust control} (DRC) is a modern umbrella term used in control theory to describe the control of dynamic systems under disturbance inputs whose distributions are uncertain. The goal is to find controllers that perform optimally under the worst-case disturbance distributions within an ambiguity set. In the past (before the term DRC was coined), this ambiguity set has, e.g., been described by constraints on the autocorrelation function of the disturbance distribution \cite{gusev1995minimax_a,gusev1995minimax_b}; a problem which has been motivated as a robust version of $H_2$ optimal control. In \cite{gusev1996method}, the author shows that optimal output feedback synthesis subject to worst-case disturbances with restricted autocorrelation function can be cast as an infinite dimensional convex optimization problem (with an auxiliary model matching problem). Subsequently, it has been shown in \cite{scherer2000robust} that DRC synthesis can be cast as a small sized SDP for restrictions on a autocorrelation sequence up to a finite lag.

The label DRC emerged later as an analog to \emph{distributionally robust optimization} (DRO). The latter studies optimization problems, where problem parameters are random variables whose distributions are assumed to lie within an ambiguity set, which is commonly defined using probability metrics \cite{mohajerin2018data,nguyen2023bridging} or the first two moments \cite{delage2010distributionally}. Wasserstein-based ambiguity sets are popular in this context due to favorable out-of-sample performance \cite{mohajerin2018data} and the ability to incorporate prior knowledge through nominal distributions.

The link to DRO has popularized Wasserstein ambiguity sets also in DRC. Consequently, a number of publications started to pursue these concepts. For example, \cite{kim2023distributional} studies state feedback DRC, noticing that state feedback policies adapt only to correlated, but not uncorrelated disturbances from Wasserstein ambiguity sets. State estimation problems, on the other hand, yield solutions that adapt to the disturbance distributions, even for independent disturbances \cite{han2024distributionally}. Most interesting is the case of output feedback policies. This is investigated in \cite{YBS25_rocond} and \cite{9992738} over a finite control horizon and in \cite{brouillon2023distributionally} over an infinite control horizon. However, the latter generalization relies on a stationarity assumption on the disturbance process. Unlike general robust controller synthesis, distributionally robust output feedback synthesis can be convexified \cite{taskesen2024distributionally} (cmp. \cite{kotsalis2021convex}). However, up to this date, convexificatoins of output feedback infinite-horizon Wasserstein DRC lack the elegant SDP formulations, as obtained for the problem of constrained autocorrelation functions in \cite{scherer2000robust}.

Finally, while we focus on linear quadratic problems, we mention works that apply DRC to slightly different settings. E.g., constrained DRC problems are investigated in \cite{van2015distributionally} and the articles \cite{10384311, 10313386, kargin2024infinite} study the gap between the cost of a causal controller and the cost of the optimal noncausal controller (regret) instead of the worst-case (expected) cost.

\section{Contributions}
\label{sec:contributions}

The present article makes the following contributions: (1) We prove that infinite-horizon DRC with Wasserstein ambiguity sets admits an exact moment relaxation, establishing tight upper and lower bounds that eliminate any relaxation gap. This exactness result enables the reduction of infinite-dimensional distributionally robust optimization to finite-dimensional convex programming. (2) We establish the first complete equivalence between infinite-horizon Wasserstein-distributionally robust control and robust $H_2$ synthesis, proving that DRC with Wasserstein ambiguity sets is precisely a structured uncertainty variant of classical robust control. (3) Unlike prior work, we develop compact LMI synthesis formulations (similar to \cite{scherer2000robust}) of Wasserstein-DRC for both temporally correlated and independent disturbances, eliminating the need for horizon-dependent discretizations or stationarity assumptions and enabling polynomial-time synthesis via standard SDP solvers. (4) The SDP formulations we propose inherit the modular structure of traditional robust control problems. As a consequence, we are able to merge the old and new results on DRC in Appendix~\ref{app:mergingAmbiguitySets}, considering uncertainties with both Wasserstein and autocorrelation restrictions simultaneously.

\section{Preliminaries}
\label{sec:preliminaries}

In this article, we make use of well-known results on the 2-Wasserstein metric and its relation to the second moment matrices of random variables. % The Wasserstein metric is defined as follows.
\begin{definition}%[Wasserstein Metric]
	\label{def:wasserstein}
	The 2-Wasserstein metric between probability distributions $\bbP_1$ and $\bbP_2$ on $\bbR^d$ is defined as
	\begin{align*}
		\bbW(\bbP_1,\bbP_2) := \inf_{\pi \in \Pi(\bbP_1,\bbP_2)} \left(\int_{\bbR^d \times \bbR^d} \|\xi_1 - \xi_2\|^2 \diff\pi(\xi_1,\xi_2)\right)^{\frac{1}{2}},
	\end{align*}
	where $\Pi(\bbP_1,\bbP_2)$ is the set of all joint distributions of the random variables $\xi_1,\xi_2$ with marginal distributions $\bbP_1,\bbP_2$.
\end{definition}

In \cite{gelbrich1990formula}, the Gelbrich distance between probability distributions is introduced. It is a lower bound on the 2-Wasserstein metric and enjoys the following properties.

\begin{theorem}[{\cite[Theorem 2.1 and Proposition 2.2]{gelbrich1990formula}}]
	\label{thm:gelbrichProperties}
	For two distributions $\bbP_1$, $\bbP_2$ on $\bbR^d$ with mean vectors $\mu_1$ and $\mu_2$ and covariance matrices $\Sigma_1$ and $\Sigma_2$, respectively, the following holds:
	\begin{enumerate}
		\item $\bbG(\bbP_1,\bbP_2) \leq \bbW(\bbP_1,\bbP_2)$, where $\bbG(\bbP_1,\bbP_2)$ denotes the Gelbrich distance between $\bbP_1$, $\bbP_2$ and is defined as
			\begin{align*}
		\bbG(\bbP_1,\bbP_2)^2 :=\| \mu_1 - \mu_2 \|^2 + \tr\bigl(\Sigma_1 + \Sigma_2 - 2 \bigl(\Sigma_1^{\frac{1}{2}} \Sigma_2 \Sigma_1^{\frac{1}{2}}\bigr)^\frac{1}{2}  \bigr).
	\end{align*}
	This bound is exact if $\bbP_1$, $\bbP_2$ are normal distributions.		
		\item If $\mu_1 = \mu_2$, the Gelbrich distance $\bbG(\bbP_1,\bbP_2)$ is equal to the optimal value of the optimization problem
		\begin{align*}
			\minimize_{\dvr{\ovl{\Sigma}} \in \bbR^{d\times d}} & \tr \big( \Sigma_1 + \Sigma_2 - 2\dvr{\ovl{\Sigma}}  \big)^{\frac{1}{2}}\\
			\mathrm{s.t.} & \begin{pmatrix}
				\Sigma_1 & \dvr{\ovl{\Sigma}}\\
				\dvr{\ovl{\Sigma}^\top} & \Sigma_2
			\end{pmatrix} \succeq 0.
		\end{align*}
		\item If $\Sigma_1 \succ 0$ and $\bbP_1$, $\bbP_2$ are normal distributions with zero mean, then there exists an optimal linear transport plan $\Delta \in \bbR^{d\times d}$ and two random variables $\xi_1 \sim \bbP_1$, $\xi_2 \sim \bbP_2$ such that $\xi_2 = \Delta \xi_1$ and
      \begin{equation*}
         \bbW(\bbP_1,\bbP_2) = \bbG(\bbP_1,\bbP_2) = \sqrt{\tr\big( (I-\Delta)  \Sigma_1 (I-\Delta)^\top \big)}.
      \end{equation*}
	\end{enumerate}
\end{theorem}

\section{Robust performance analysis}
\label{sec:robustAnalysis}

As is customary in robust control, we begin with the \emph{robust performance analysis} problem, i.e., with the study of \eqref{eq:finiteHorizonProblem} for fixed controller parameters.

We first note that the worst-case disturbance problem \eqref{eq:finiteHorizonProblem} can be rewritten as
\begin{subequations}
  \label{eq:worstCaseProblemAlternative}
\begin{align}
  \maximize_{\substack{w: \bbN_0 \to L_2(\Omega,\bbR^{n_w}),\\\hat{w}: \bbN_0 \to L_2(\Omega,\bbR^{n_w})}} ~&~ \limsup_{T \to \infty} \frac{1}{T}\sum_{t = 0}^{T-1} \bbE \left\| z(t) \right\|^2, \label{eq:worstCaseProblemAlternativeA}\\
  &~ \begin{pmatrix}
    \chi(t+1)\\
    z(t)
  \end{pmatrix}
  =
  \begin{pmatrix}
    \dvrC{\calA} & \dvrC{\calB}\\
    \dvrC{\calC} & \dvrC{\calD}
  \end{pmatrix}
  \begin{pmatrix}
    \chi(t)\\
    w(t)
  \end{pmatrix}, \label{eq:worstCaseProblemAlternativeB}\\
  &~ \gamma^2 \geq \bbE \left\| w(t) - \hat{w}(t) \right\|^2, \label{eq:worstCaseProblemAlternativeC}\\
  &~ \bbP_{\hat{w}(t)} = \bbP_\text{nom}, \label{eq:worstCaseProblemAlternativeD}\\
	&~ \chi(0) = 0, \label{eq:worstCaseProblemAlternativeE}
\end{align}
\end{subequations}
with potentially additional constraints for independence and zero mean assumptions on the random variable $w$.

Problem \eqref{eq:worstCaseProblemAlternative} introduces an auxiliary stochastic process $\hat{w}$ with marginal distributions $\bbP_{\hat{w}(t)} = \bbP_\text{nom}$ for all $t \in \bbN_0$. The purpose of $\hat{w}$ is to encode the Wasserstein constraint $\bbW(\bbP_{w(t)},\bbP_\text{nom}) \leq \gamma$ in the conditions \eqref{eq:worstCaseProblemAlternativeC} and \eqref{eq:worstCaseProblemAlternativeD}. This reformulation exploits the variational definition of the Wasserstein distance: by  requiring $\bbE\|w(t) - \hat{w}(t)\|^2 \leq \gamma^2$, the joint distribution of $(w(t), \hat{w}(t))$ provides an explicit coupling that certifies the Wasserstein bound. Since any feasible coupling yields an upper bound on the Wasserstein distance, this reformulation captures distributions within the Wasserstein ball.
This is made precise in the following result.
\begin{theorem}
  \label{thm:finiteHorizonProblemEquivalent}
  Consider a fixed controller \eqref{eq:controller} and a fixed system \eqref{eq:finiteHorizonProblemC}. Then the problems \eqref{eq:finiteHorizonProblem} and \eqref{eq:worstCaseProblemAlternative} are equivalent in the sense that every feasible point $w \in \scrW$ of \eqref{eq:finiteHorizonProblem} is in one-to-one correspondence with a feasible point $(w,\hat{w})$ of \eqref{eq:worstCaseProblemAlternative} and the objective values of these solutions are equal, respectively.
\end{theorem}
The proof is deferred to Appendix \ref{app:equivalence}.

At this point, we notice that \eqref{eq:worstCaseProblemAlternative} is a quadratic optimal control problem over random variables, as studied in \cite{gramlich2024moment}. In \cite{gramlich2024moment}, this problem is approached with the averaged second moment matrices $\Sigma(T)$, which are defined to be equal to
\begin{align*}
  \begin{pmatrix}
    \Sigma_{\chi\chi}(T) & \Sigma_{\chi w}(T) & \Sigma_{\chi \hat{w}}(T)\\
    \Sigma_{w\chi}(T) & \Sigma_{ww}(T) & \Sigma_{w \hat{w}}(T)\\
    \Sigma_{\hat{w}\chi}(T) & \Sigma_{\hat{w}w}(T) & \Sigma_{\hat{w} \hat{w}}(T)
  \end{pmatrix}
  :=
  \frac{1}{T}\sum_{t=0}^{T-1}
  \bbE
  \begin{pmatrix}
    \chi(t)\\
    w(t)\\
    \hat{w}(t)
  \end{pmatrix}
  \begin{pmatrix}
    \chi(t)\\
    w(t)\\
    \hat{w}(t)
  \end{pmatrix}^\top \!\!,
\end{align*}
and a stationary moment matrix $\dvrMax{\Sigma}$, where the latter is supposed to represent the limit of $\Sigma(T)$ as $T \to \infty$ if that limit exists.
Using the stationary moment matrix $\dvrMax{\Sigma}$, we define the moment relaxed optimization problem
\begin{subequations}
  \label{eq:momentRelaxedProblem}
\begin{align}
  \maximize_{\dvrMax{\Sigma} \in \scrS} ~&~ \tr\left(
  \begin{pmatrix}
    \dvrC{\calC} & \dvrC{\calD} & 0
  \end{pmatrix}
  \dvrMax{\Sigma}
  \begin{pmatrix}
    \dvrC{\calC} & \dvrC{\calD} & 0
  \end{pmatrix}^\top
  \right), \label{eq:momentRelaxedProblemA}
\end{align}
subject to
\begin{align}
  &~
  % \dvrMax{\Sigma_{\chi\chi}}
  \begin{pmatrix}
    I & 0 & 0
  \end{pmatrix}
  \dvrMax{\Sigma}
  \begin{pmatrix}
    I & 0 & 0
  \end{pmatrix}^\top
  =
  \begin{pmatrix}
    \dvrC{\calA} & \dvrC{\calB} & 0
  \end{pmatrix}
  \dvrMax{\Sigma}
  \begin{pmatrix}
    \dvrC{\calA} & \dvrC{\calB} & 0
  \end{pmatrix}^\top,\label{eq:momentRelaxedProblemB}\\
  &~
  \gamma^2 \geq
  \tr\left(
  \begin{pmatrix}
    0 & I & 0\\
    0 & 0 & I
  \end{pmatrix}
  \dvrMax{\Sigma}
  \begin{pmatrix}
    0 & I & 0\\
    0 & 0 & I\\
  \end{pmatrix}^\top
  \begin{pmatrix}
    I & -I\\
    -I & I
  \end{pmatrix}
  \right) , \label{eq:momentRelaxedProblemC}\\
  &~ \Sigma_\text{nom} = \begin{pmatrix}
    0 & 0 & I
  \end{pmatrix}
  \dvrMax{\Sigma}
  \begin{pmatrix}
    0 & 0 & I
  \end{pmatrix}^\top, \label{eq:momentRelaxedProblemE}\\
  &~ \dvrMax{\Sigma} \succeq 0. \label{eq:momentRelaxedProblemD}
\end{align}
\end{subequations}
Here, the set $\scrS \subsEq \{\scrS_\mathrm{cor}, \scrS_\mathrm{ind}\}$ incorporates the independence and zero mean assumptions on $\scrW_\mathrm{cor}$ and $\scrW_{\mathrm{ind}}$. The matrices $\dvrMax{\Sigma}$ are constrained to be of the form
\begin{align}
  \begin{pmatrix}
    \dvrMax{\Sigma_{\chi\chi}} & \dvrMax{\Sigma_{\chi w}} & \dvrMax{\Sigma_{\chi\hat{w}}}\\
    \dvrMax{\Sigma_{w\chi}} & \dvrMax{\Sigma_{ww}} & \dvrMax{\Sigma_{w\hat{w}}}\\
    \dvrMax{\Sigma_{\hat{w}\chi}} & \dvrMax{\Sigma_{\hat{w}w}} & \dvrMax{\Sigma_{\hat{w}\hat{w}}}
  \end{pmatrix},
  &&
  \begin{pmatrix}
    \dvrMax{\Sigma_{\chi\chi}} & 0 & 0\\
    0 & \dvrMax{\Sigma_{ww}} & \dvrMax{\Sigma_{w\hat{w}}}\\
    0 & \dvrMax{\Sigma_{\hat{w}w}} & \dvrMax{\Sigma_{\hat{w}\hat{w}}}
  \end{pmatrix}, \label{eq:SigmaPartitioning}
\end{align}
for $\dvrMax{\Sigma} \in \scrS_\mathrm{cor}$ or $\dvrMax{\Sigma} \in \scrS_\mathrm{ind}$, respectively. The matrix $\dvrMax{\Sigma_{\chi w}}$ is zero in the case of $\dvrMax{\Sigma} \in \scrS_\text{ind}$, because $\chi(t)$ can be expressed as a linear function of $w(t-1),\ldots,w(0)$ and the latter are, by assumption, zero mean and independent from $w(t)$. Further, $\dvrMax{\Sigma_{\chi \hat{w}}}$ is zero, because $\hat{w}(t)$ is supposed to capture the coupling distribution of the Wasserstein constraint on $\bbP_{w(t)}$ and $\bbP_\text{nom}$ and $w(t)$ is independent of $\chi(t)$ in the case of $\scrW_\mathrm{ind}$.

The transition from \eqref{eq:worstCaseProblemAlternative} to \eqref{eq:momentRelaxedProblem} is based on a moment relaxation of the Wasserstein conditions \eqref{eq:worstCaseProblemAlternativeC} and \eqref{eq:worstCaseProblemAlternativeD}. A comparison of \eqref{eq:momentRelaxedProblemC} and \eqref{eq:momentRelaxedProblemD} with 2) in Theorem~\ref{thm:gelbrichProperties} reveals that this moment relaxation replaces the Wasserstein metric with the Gelbrich distance. At this point, related literature utilizes the fact Theorem~\ref{thm:gelbrichProperties} 1), that the Gelbrich distance equals the Wasserstein distance for normal distributions (compare \cite{kuhn2019wasserstein}). However, since \eqref{eq:finiteHorizonProblem} is an infinite-dimensional problem and the moment relaxation \eqref{eq:momentRelaxedProblem} involves averaging over time, it is necessary to investigate the equivalence of \eqref{eq:worstCaseProblemAlternative} and \eqref{eq:momentRelaxedProblem} in more detail.

To show equivalence of \eqref{eq:worstCaseProblemAlternative} and \eqref{eq:momentRelaxedProblem} also in our scenario, we establish two propositions.

\begin{proposition}
	\label{prop:finiteHorizonUpperBound}
	Let $\dvrC{\calA}$ be stable and $(\scrW,\scrS) \in \{(\scrW_\mathrm{cor},\scrS_\mathrm{cor}), (\scrW_\mathrm{ind},\scrS_\mathrm{ind})\}$. %Then the following statements hold.
  Then, for any sequence of random variables $w \in \scrW$ and any $T \in \bbN$, there exists a moment matrix $\dvrMax{\Sigma} \in \scrS$ satisfying \eqref{eq:momentRelaxedProblemB} to \eqref{eq:momentRelaxedProblemD} with
  \begin{align*}
    \dvrMax{\Sigma_{\chi\chi}} \succeq \Sigma_{\chi\chi}(T),\dvrMax{\Sigma_{\chi w}} = \Sigma_{\chi w}(T), \dvrMax{\Sigma_{\chi \hat{w}}} = \Sigma_{\chi \hat{w}}(T),\\
    \dvrMax{\Sigma_{ww}} = \Sigma_{ww}(T),
    \dvrMax{\Sigma_{w\hat{w}}} = \Sigma_{w\hat{w}}(T),
     \dvrMax{\Sigma_{\hat{w}\hat{w}}} = \Sigma_{\hat{w}\hat{w}}(T).
    %  \dvrMax{\sigma_\chi} = \sigma_\chi(T), \dvrMax{\sigma_w} = \sigma_w(T), \dvrMax{\sigma_{\hat{w}}} = \sigma_{\hat{w}}(T).
  \end{align*}
\end{proposition}

\begin{proof}
  We first note that problem \eqref{eq:finiteHorizonProblem} is always feasible.
	The proposition is proven by choosing the candidate sub-matrices
	\begin{align*}
		\dvrMax{\Sigma_{\chi w}} &= \Sigma_{\chi w}(T), \ldots, 
    \dvrMax{\Sigma_{\hat{w}\hat{w}}} = \Sigma_{\hat{w}\hat{w}}(T)
    % \dvrMax{\sigma_{\hat{w}}} = \sigma_{\hat{w}}(T),
	\end{align*}
	and obtaining the candidate for $\dvrMax{\Sigma_{\chi\chi}}$ as the solution of the Lyapunov equation
	\begin{align}
		\dvrMax{\Sigma_{\chi\chi}} = 
		\begin{pmatrix}
			\dvrC{\calA^\top}\\
			\dvrC{\calB^\top}
		\end{pmatrix}^\top 
		\begin{pmatrix}
			\dvrMax{\Sigma_{\chi\chi}} & \Sigma_{\chi w}(T)\\
			\Sigma_{w \chi}(T) & \Sigma_{ww}(T)
		\end{pmatrix}
		\begin{pmatrix}
			\dvrC{\calA^\top}\\
			\dvrC{\calB^\top}
		\end{pmatrix}. \label{eq:Lyapunov1}
	\end{align}
	Clearly, as $\chi(0) = 0$, $\Sigma_{\chi\chi}(T), \Sigma_{\chi w}(T), \Sigma_{ww}(T)$ satisfy the Lyapunov equation
	\begin{align}
		\frac{T+1}{T}\Sigma_{\chi\chi}(T+1) = 
		\begin{pmatrix}
			\dvrC{\calA^\top}\\
			\dvrC{\calB^\top}
		\end{pmatrix}^\top 
		\begin{pmatrix}
			\Sigma_{\chi\chi}(T) & \Sigma_{\chi w}(T)\\
			\Sigma_{w \chi}(T) & \Sigma_{ww}(T)
		\end{pmatrix}
		\begin{pmatrix}
			\dvrC{\calA^\top}\\
			\dvrC{\calB^\top}
		\end{pmatrix}. \label{eq:Lyapunov2}
	\end{align}
	By subtracting \eqref{eq:Lyapunov2} from \eqref{eq:Lyapunov1}, we obtain
	\begin{align*}
		\dvrMax{\Sigma_{\chi\chi}} - \frac{T+1}{T}\Sigma_{\chi\chi}(T+1) = \dvrC{\calA} (\dvrMax{\Sigma_{\chi\chi}} - \Sigma_{\chi\chi}(T)) \dvrC{\calA^\top} .
	\end{align*}
	Using $\frac{T+1}{T}\Sigma_{\chi\chi}(T+1) = \Sigma_{\chi\chi}(T) + \frac{1}{T} \bbE \chi(T)\chi(T)^\top$ yields
	\begin{align*}
		\dvrMax{\Sigma_{\chi\chi}} - \Sigma_{\chi\chi}(T) - \dvrC{\calA} (\dvrMax{\Sigma_{\chi\chi}} - \Sigma_{\chi\chi}(T)) \dvrC{\calA^\top} = \frac{1}{T} \bbE \chi(T)\chi(T)^\top \succeq 0.
	\end{align*}
	This Lyapunov inequality with the stable matrix $\dvrC{\calA}$ implies $\dvrMax{\Sigma_{\chi\chi}} \succeq \Sigma_{\chi\chi}(T)$.
  Next, we argue that $\dvrMax{\Sigma}$ satisfies the constraints \eqref{eq:momentRelaxedProblemB} to \eqref{eq:momentRelaxedProblemD}. The constraint \eqref{eq:momentRelaxedProblemB} is satisfied by the construction of $\dvrMax{\Sigma_{\chi\chi}}$. The constraint \eqref{eq:momentRelaxedProblemC} is satisfied, since $\dvrMax{\Sigma_{ww}} = \Sigma_{ww}(T)$, $\dvrMax{\Sigma_{w\hat{w}}} = \Sigma_{w\hat{w}}(T)$, and $\dvrMax{\Sigma_{\hat{w}\hat{w}}}= \Sigma_{\hat{w}\hat{w}}(T)$. The constraint \eqref{eq:momentRelaxedProblemD} holds as $\dvrMax{\Sigma_{\chi\chi}} \succeq \Sigma_{\chi\chi}(T)$ and $\Sigma(T) \succeq 0$. Also, it is clear that \eqref{eq:momentRelaxedProblemE} holds, because $\Sigma_\text{nom} = \Sigma_{\hat{w}\hat{w}}(T) = \dvrMax{\Sigma_{\hat{w}\hat{w}}}$.
\end{proof}

\begin{proposition}
  \label{prop:finiteHorizonLowerBound}
  Let $\dvrC{\calA}$ be stable and let $\dvrMax{\Sigma} \in \scrS$ satisfy \eqref{eq:momentRelaxedProblemB} to \eqref{eq:momentRelaxedProblemD}. Now perform the following construction:
  \begin{enumerate}
    \item Construct matrices $F_{w\chi},F_{\hat{w}\chi}$, $F_{wd}$, and $F_{\hat{w}d}$ such that
    \begin{align*}
      \begin{pmatrix}
        F_{w\chi}\\
        F_{\hat{w}\chi}
      \end{pmatrix}
      &=
      \begin{pmatrix}
        \dvrMax{\Sigma_{w\chi}}\\
        \dvrMax{\Sigma_{\hat{w}\chi}}
      \end{pmatrix}
      \dvrMax{\Sigma_{\chi\chi}^\dagger},
      % \begin{pmatrix}
      %   \dvrMax{\Sigma_{\chi\chi}} & \dvrMax{\sigma_\chi}\\
      %   \dvrMax{\sigma_\chi^\top} & 1
      % \end{pmatrix}^\dagger,
      \\
      \begin{pmatrix}
        F_{wd}\\
        F_{\hat{w}d}
      \end{pmatrix}
      \begin{pmatrix}
        F_{wd}\\
        F_{\hat{w}d}
      \end{pmatrix}^\top
      &=
      \begin{pmatrix}
        \dvrMax{\Sigma_{w w}} & \dvrMax{\Sigma_{w\hat{w}}}\\
        \dvrMax{\Sigma_{\hat{w}w}} & \dvrMax{\Sigma_{\hat{w}\hat{w}}}
      \end{pmatrix}
      -
      \begin{pmatrix}
        F_{w\chi}\\
        F_{\hat{w}\chi}
      \end{pmatrix}
      \begin{pmatrix}
        \dvrMax{\Sigma_{\chi w}} & \dvrMax{\Sigma_{\chi\hat{w}}}
      \end{pmatrix}.
      % \begin{pmatrix}
      %   F_{w\chi} & f_{w}\\
      %   F_{\hat{w}\chi} & f_{\hat{w}}
      % \end{pmatrix}
      % \begin{pmatrix}
      %   \dvrMax{\Sigma_{\chi w}} & \dvrMax{\Sigma_{\chi\hat{w}}}\\
      %   \dvrMax{\sigma}_w^\top & \dvrMax{\sigma}_{\hat{w}}^\top
      % \end{pmatrix}
    \end{align*}
    % and define the feedback policy
    % \begin{align*}
    %   \begin{pmatrix}
    %     w(t)\\
    %     \hat{w}(t)
    %   \end{pmatrix}
    %   =
    %   \begin{pmatrix}
    %     F_{w\chi}\\
    %     F_{\hat{w}\chi}
    %   \end{pmatrix}
    %   \chi(t)
    %   % \begin{pmatrix}
    %   %   F_{w\chi} & f_{w}\\
    %   %   F_{\hat{w}\chi} & f_{\hat{w}}
    %   % \end{pmatrix}
    %   % \begin{pmatrix}
    %   %   \chi(t)\\
    %   %   1
    %   % \end{pmatrix}
    %   +
    %   \begin{pmatrix}
    %     F_{wd}\\
    %     F_{\hat{w}d}
    %   \end{pmatrix}
    %   d(t),
    % \end{align*}
    % where $(d(t))_{t\in\bbN_0}$ is a sequence of i.i.d. standard normal random variables.
    \item Define recursively the stochastic processes $w$ and $\hat{w}$ by
    \begin{subequations}
      \label{eq:worstCasePolicy}
    \begin{align}
      \hat{\chi}(0) &\sim \calN(0,\dvrMax{\Sigma_{\chi\chi}}), \label{eq:worstCasePolicyA}\\
      d(t) &\sim \calN(0,I), & t \in \bbN_0,\\
      \begin{pmatrix}
        w(t)\\
        \hat{w}(t)
      \end{pmatrix}
      &=
      \begin{pmatrix}
        F_{w\chi}\\
        F_{\hat{w}\chi}
      \end{pmatrix}
      \hat{\chi}(t)
      % \begin{pmatrix}
      %   F_{w\chi} & f_{w}\\
      %   F_{\hat{w}\chi} & f_{\hat{w}}
      % \end{pmatrix}
      % \begin{pmatrix}
      %   \chi(t)\\
      %   1
      % \end{pmatrix}
      +
      \begin{pmatrix}
        F_{wd}\\
        F_{\hat{w}d}
      \end{pmatrix}
      d(t)
      & t \in \bbN_0, \label{eq:worstCasePolicyC}\\
      \hat{\chi}(t+1) &=
      \dvrC{\calA} \hat{\chi}(t) + \dvrC{\calB} w(t) \label{eq:worstCasePolicyD},
      & t \in \bbN_0,
    \end{align}
  \end{subequations}
with the auxiliary random variables $\hat{\chi}(t)$ and $d(t)$ for $t \in \bbN_0$, where $(d(t))_{t \in \bbN_0}, \hat{\chi}(0)$ are assumed to be independent.
  \end{enumerate}
  Then, the following statements hold.
  \begin{enumerate}
    \item The stochastic processes $w$ and $\hat{w}$ satisfies \eqref{eq:worstCaseProblemAlternativeC}, \eqref{eq:worstCaseProblemAlternativeD} and $w \in \scrW$.
    \item The objective value of $(w,\hat{w})$ for the problem \eqref{eq:worstCaseProblemAlternative} is the same as the objective value of $\dvrMax{\Sigma}$ for \eqref{eq:momentRelaxedProblem}.
  \end{enumerate}
\end{proposition}
\begin{proof}
  \textbf{Step 1:} We show that the joint moment matrix of $\hat{\chi}(t),w(t),\hat{w}(t)$ equals $\dvrMax{\Sigma}$.

  Substituting \eqref{eq:worstCasePolicyC} for $w(t)$ and $\hat{w}(t)$ reveals the relation
  \begin{align*}
    \bbE \begin{pmatrix}
      \hat{\chi}(t)\\
      w(t)\\
      \hat{w}(t)
    \end{pmatrix}
    \begin{pmatrix}
      \hat{\chi}(t)\\
      w(t)\\
      \hat{w}(t)
    \end{pmatrix}^\top
    &=
    \begin{pmatrix}
      I\\
      F_{w\chi}\\
      F_{\hat{w}\chi}
    \end{pmatrix}
    \bbE
    \hat{\chi}(t)\hat{\chi}(t)^\top
    \begin{pmatrix}
      I\\
      F_{w\chi}\\
      F_{\hat{w}\chi}
    \end{pmatrix}^\top\\
    &+
    \begin{pmatrix}
      0\\
      F_{wd}\\
      F_{\hat{w}d}
    \end{pmatrix}
    \begin{pmatrix}
      0\\
      F_{wd}\\
      F_{\hat{w}d}
    \end{pmatrix}^\top
  \end{align*}
  for the joint moment matrix of $\hat{\chi}(t),w(t),\hat{w}(t)$. Here, we used $\bbE \hat{\chi}^\top(t)d(t) = 0$, which follows from $\hat{\chi}(t)$ being a function of $d(t-1),\ldots,d(0)$ and $\hat{\chi}(0)$ due to the recursive definition \eqref{eq:worstCasePolicyD}.
  
  Next, for an induction argument, assume that $\bbE \hat{\chi}(t)\hat{\chi}(t)^\top = \dvrMax{\Sigma_{\chi\chi}}$. Then,
  \begin{align}
    \bbE \begin{pmatrix}
      \hat{\chi}(t)\\
      w(t)\\
      \hat{w}(t)
    \end{pmatrix}
    \begin{pmatrix}
      \hat{\chi}(t)\\
      w(t)\\
      \hat{w}(t)
    \end{pmatrix}^\top
    &=
    \begin{pmatrix}
      \dvrMax{\Sigma_{\chi\chi}}\\
      F_{w\chi}\dvrMax{\Sigma_{\chi\chi}}\\
      F_{\hat{w}\chi}\dvrMax{\Sigma_{\chi\chi}}
    \end{pmatrix}
    \begin{pmatrix}
      I\\
      F_{w\chi}\\
      F_{\hat{w}\chi}
    \end{pmatrix}^\top
    +
    \begin{pmatrix}
      0\\
      F_{wd}\\
      F_{\hat{w}d}
    \end{pmatrix}
    \begin{pmatrix}
      0\\
      F_{wd}\\
      F_{\hat{w}d}
    \end{pmatrix}^\top . \label{eq:worstCasePolicyMoment}
  \end{align}
  We point out that positive semi-definiteness of $\dvrMax{\Sigma}$ implies $\ker \dvrMax{\Sigma_{\chi\chi}} \subsEq \ker\dvrMax{\Sigma_{w\chi}}$ and $\ker \dvrMax{\Sigma_{\chi\chi}} \subsEq \ker\dvrMax{\Sigma_{\hat{w}\chi}}$. Consequently, we infer
  \begin{align*}
    F_{w\chi}\dvrMax{\Sigma_{\chi\chi}} = \dvrMax{\Sigma_{w\chi}}, \quad F_{\hat{w}\chi}\dvrMax{\Sigma_{\chi\chi}} = \dvrMax{\Sigma_{\hat{w}\chi}}
  \end{align*}
  by 1) in Proposition~\ref{prop:finiteHorizonLowerBound}. Therefore, \eqref{eq:worstCasePolicyMoment} reads
  \begin{align*}
    \bbE \begin{pmatrix}
      \hat{\chi}(t)\\
      w(t)\\
      \hat{w}(t)
    \end{pmatrix}
    \begin{pmatrix}
      \hat{\chi}(t)\\
      w(t)\\
      \hat{w}(t)
    \end{pmatrix}^\top
    &=
    \begin{pmatrix}
      \dvrMax{\Sigma_{\chi\chi}}\\
      \dvrMax{\Sigma_{w\chi}}\\
      \dvrMax{\Sigma_{\hat{w}\chi}}
    \end{pmatrix}
    \begin{pmatrix}
      I\\
      F_{w\chi}\\
      F_{\hat{w}\chi}
    \end{pmatrix}^\top
    +
    \begin{pmatrix}
      0\\
      F_{wd}\\
      F_{\hat{w}d}
    \end{pmatrix}
    \begin{pmatrix}
      0\\
      F_{wd}\\
      F_{\hat{w}d}
    \end{pmatrix}^\top ,
  \end{align*}
  which is equal to $\dvrMax{\Sigma}$. Next, recalling the recursive definition \eqref{eq:worstCasePolicyD} of $\hat{\chi}$, we obtain
  \begin{align*}
    \bbE \hat{\chi}(t+1) \hat{\chi}(t+1)^\top = \begin{pmatrix}
      \dvrC{\calA} & \dvrC{\calB} & 0
    \end{pmatrix} \dvrMax{\Sigma} \begin{pmatrix}
      \dvrC{\calA} & \dvrC{\calB} & 0
    \end{pmatrix}^\top \overset{\eqref{eq:momentRelaxedProblemB}}{=} \dvrMax{\Sigma_{\chi\chi}}.
  \end{align*}
  Consequently, with \eqref{eq:worstCasePolicyA} being the start of an induction argument on $t$, it follows that $\bbE \hat{\chi}(t) \hat{\chi}(t)^\top = \dvrMax{\Sigma_{\chi\chi}}$ for all $t \in \bbN_0$. This completes the proof of Step 1.

  \noindent \textbf{Step 2:} The stochastic processes $w,\hat{w}$ satisfy \eqref{eq:worstCaseProblemAlternativeC} and \eqref{eq:worstCaseProblemAlternativeD}.

  Since the joint moment matrix of $\hat{\chi}(t),w(t),\hat{w}(t)$ equals $\dvrMax{\Sigma}$, the joint moment matrix of $w(t),\hat{w}(t)$ is given by
  \begin{align*}
    \begin{pmatrix}
      \dvrMax{\Sigma_{ww}} & \dvrMax{\Sigma_{w\hat{w}}}\\
      \dvrMax{\Sigma_{\hat{w}w}} & \dvrMax{\Sigma_{\hat{w}\hat{w}}}
    \end{pmatrix}.
  \end{align*}
  Consequently, \eqref{eq:momentRelaxedProblemC} implies that the stochastic processes $w$ and $\hat{w}$ satisfy \eqref{eq:worstCaseProblemAlternativeC}. In addition, each $\hat{w}(t)$ is a zero mean normal random variable by construction. Hence, it has second moment matrix $\dvrMax{\Sigma_{\hat{w}\hat{w}}}$ and satisfies \eqref{eq:worstCaseProblemAlternativeD}.

  \noindent \textbf{Step 3:} Apply $w$ to the closed-loop \eqref{eq:finiteHorizonProblemC}. Then the moments of $\chi(t) - \hat{\chi}(t)$ converge to zero exponentially as $t \to \infty$.

  If $(w,\hat{w})$ is employed in \eqref{eq:worstCaseProblemAlternative}, a stochastic process $\chi$ is generated. By subtracting \eqref{eq:worstCasePolicyD} from \eqref{eq:worstCaseProblemAlternativeB}, we obtain
  \begin{align*}
    \chi(t+1) - \hat{\chi}(t+1) = \dvrC{\calA} (\chi(t) - \hat{\chi}(t)).
  \end{align*}
  Since $\dvrC{\calA}$ is stable and $\chi(0) - \hat{\chi}(0)$ is Gaussian (by \eqref{eq:worstCasePolicyA}, \eqref{eq:worstCaseProblemAlternativeE}), $\chi(t) - \hat{\chi}(t)$ converges to zero for any initial condition and its second moment also converges to zero exponentially as $t \to \infty$.

  \noindent \textbf{Step 4:} The objective value of $(w,\hat{w})$ for the problem \eqref{eq:worstCaseProblemAlternative} is the same as the objective value of $\dvrMax{\Sigma}$ for the problem \eqref{eq:momentRelaxedProblem}.

  Since the joint moment matrix of $\hat{\chi}(t),w(t),\hat{w}(t)$ equals $\dvrMax{\Sigma}$ and $\chi(t)-\hat{\chi}(t)$ and its second moment converge to zero exponentially as $t \to \infty$, the joint moment matrix of $\chi(t),w(t),\hat{w}(t)$ converges to $\dvrMax{\Sigma}$ exponentially. In this case, the asymptotic averaged cost \eqref{eq:worstCaseProblemAlternativeA} is
  \begin{align*}
   \frac{1}{T}\sum_{t = 0}^{T-1} \bbE \left\| z(t) \right\|^2
    &=
    \tr\left(
  \begin{pmatrix}
    \dvrC{\calC} & \dvrC{\calD} & 0
  \end{pmatrix}
  \Sigma(T)
  \begin{pmatrix}
    \dvrC{\calC} & \dvrC{\calD} & 0
  \end{pmatrix}^\top
  \right)\\
  &\overset{T \to \infty}{=}
  \tr\left(
  \begin{pmatrix}
    \dvrC{\calC} & \dvrC{\calD} & 0
  \end{pmatrix}
  \dvrMax{\Sigma}
  \begin{pmatrix}
    \dvrC{\calC} & \dvrC{\calD} & 0
  \end{pmatrix}^\top
  \right).
  \end{align*}
\end{proof}

With Proposition~\ref{prop:finiteHorizonUpperBound} and Proposition~\ref{prop:finiteHorizonLowerBound}, we can now prove the equivalence between \eqref{eq:worstCaseProblemAlternative} and \eqref{eq:momentRelaxedProblem}.

\begin{theorem}
  \label{thm:exactnessMomentRelaxation}
  Let $\dvrC{\calA}$ be stable. Then, the optimization problems \eqref{eq:worstCaseProblemAlternative} and \eqref{eq:momentRelaxedProblem} are feasible and their optimal values coincide.
\end{theorem}
\begin{proof}
  We first point out that given a solution $\dvrMax{\Sigma_{\chi\chi}}$ to the Lyapunov equation $\dvrMax{\Sigma_{\chi\chi}} = \dvrC{\calA}\dvrMax{\Sigma_{\chi\chi}} \dvrC{\calA^\top} + \dvrC{\calB}\Sigma_{\text{nom}}\dvrC{\calB^\top}$, $\dvrMax{\Sigma} = \diag\left(\dvrMax{\Sigma_{\chi\chi}}, \begin{pmatrix} \Sigma_\text{nom} & \Sigma_\text{nom}\\ \Sigma_\text{nom} & \Sigma_\text{nom} \end{pmatrix}\right)$ is a feasible candidate for \eqref{eq:momentRelaxedProblem}. Moreover, if $\dvrC{\calA}$ is stable, then $w = \hat{w}$ with $\hat{w}(t) \sim \bbP_\text{nom}$ i.i.d. for all $t \in \bbN_0$ is a feasible candidate for \eqref{eq:worstCaseProblemAlternative}.

  Next, if $(\chi,w,\hat{w})$ is a feasible candidate for \eqref{eq:worstCaseProblemAlternative}, then, by Proposition~\ref{prop:finiteHorizonUpperBound}, every member of the sequence $\Sigma(T)$ of averaged moments is bounded from above by a stationary moment matrix $\dvrMax{\Sigma}$ that satisfies \eqref{eq:momentRelaxedProblemB} to \eqref{eq:momentRelaxedProblemD}. Since the stochastic process $(\chi,w,\hat{w})$ is an arbitrary feasible candidate of \eqref{eq:worstCaseProblemAlternative}, this implies that the optimal value of \eqref{eq:momentRelaxedProblem} is larger than or equal to the optimal value of \eqref{eq:worstCaseProblemAlternative}.
  
  On the other hand, if $\dvrMax{\Sigma}$ is a feasible candidate for \eqref{eq:momentRelaxedProblem}, then Proposition~\ref{prop:finiteHorizonLowerBound} establishes the existence of a feasible stochastic process $(\chi,w,\hat{w})$ with the same value for \eqref{eq:worstCaseProblemAlternative} as $\dvrMax{\Sigma}$ has for \eqref{eq:momentRelaxedProblem}. This implies that the optimal value of \eqref{eq:worstCaseProblemAlternative} is larger than or equal to the optimal value of \eqref{eq:momentRelaxedProblem}.
\end{proof}

According to Theorem~\ref{thm:exactnessMomentRelaxation}, we can utilize the moment relaxation problem \eqref{eq:momentRelaxedProblem} to exactly analyze the worst-case performance of a controller \eqref{eq:finiteHorizonProblemC}. However, if we were willing to minimize this worst-case performance over the controller parameters, we would be confronted with a minimax optimization problem, which is notoriously hard to solve. For this reason, we study the Lagrangian dual of \eqref{eq:momentRelaxedProblem}, which is a minimization problem, in the next section.

\subsection{Robust performance analysis for correlated disturbances}
\label{subsec:robustPerformanceAnalysisCor}

In this section, we derive the dual problem of problem \eqref{eq:momentRelaxedProblem} for the case of correlated disturbances. This dual takes the form of a robust performance analysis problem typically obtained from dissipativity theory.

\begin{theorem}
  \label{thm:performanceAnalysisCor}
  Let $\dvrC{\calA}$ be stable. Then, for the case $\scrS = \scrS_\text{cor}$, \eqref{eq:momentRelaxedProblem} is equivalent to the optimization problem
  \begin{subequations}
    \label{eq:performanceAnalysisCor}
  \begin{align}
    &\minimize_{\dvr{\lambda} \geq 0,\dvr{\calP}, \dvr{Q}} ~~ \tr\left( \dvr{Q} \Sigma_\text{nom} \right) + \dvr{\lambda}\gamma^2 \label{eq:performanceAnalysisCorA}\\
    &\mathrm{s.t.} ~~ 
    % \begin{pmatrix}
    %   I & 0 & 0\\
    %   \dvrC{\calA} & \dvrC{\calB} & 0\\
    %   \dvrC{\calC} & \dvrC{\calD} & 0\\
    %   0 & I & 0\\
    %   0 & 0 & I
    % \end{pmatrix}^\top 
    \begin{pmatrix}
      \bullet
    \end{pmatrix}^\top
    \begin{pmatrix}
      -\dvr{\calP} & 0 & 0 & 0 & 0\\
      0 & \dvr{\calP} & 0 & 0 & 0\\
      0 & 0 & I & 0 & 0\\
      0 & 0 & 0 & -\dvr{\lambda}I & \dvr{\lambda}I\\
      0 & 0 & 0 & \dvr{\lambda}I & -\dvr{Q} - \dvr{\lambda}I
    \end{pmatrix}
    \begin{pmatrix}
      I & 0 & 0\\
      \dvrC{\calA} & \dvrC{\calB} & 0\\
      \dvrC{\calC} & \dvrC{\calD} & 0\\
      0 & I & 0\\
      0 & 0 & I
    \end{pmatrix} \prec 0. \label{eq:performanceAnalysisCorB}
  \end{align}
\end{subequations}
Moreover, the stability of $\dvrC{\calA}$ guarantees the feasibility of \eqref{eq:performanceAnalysisCor}.
\end{theorem}

\begin{proof}
  \noindent \textbf{Step 1:} Lagrangian duality. 
  
  The dual optimization problem corresponding to \eqref{eq:performanceAnalysisCor} is
  \begin{align*}
    &\sup_{\dvrMax{\Sigma} \succeq 0} \inf_{\dvr{\lambda} \geq 0, \dvr{Q},\dvr{\calP}} ~\tr\left( \dvr{Q} \Sigma_\text{nom} \right) + \dvr{\lambda}\gamma^2 +\\
    &
    \tr
    \Bigg( 
      % \begin{pmatrix}
      %   I & 0 & 0\\
      %   \dvrC{\calA} & \dvrC{\calB} & 0\\
      %   \dvrC{\calC} & \dvrC{\calD} & 0\\
      %   0 & I & 0\\
      %   0 & 0 & I
      % \end{pmatrix}^\top 
      \begin{pmatrix}
        \bullet
      \end{pmatrix}^\top
      \begin{pmatrix}
        -\dvr{\calP} & 0 & 0 & 0 & 0\\
        0 & \dvr{\calP} & 0 & 0 & 0\\
        0 & 0 & I & 0 & 0\\
        0 & 0 & 0 & -\dvr{\lambda}I & \dvr{\lambda}I\\
        0 & 0 & 0 & \dvr{\lambda}I & -\dvr{Q} - \dvr{\lambda}I
      \end{pmatrix}
      \begin{pmatrix}
        I & 0 & 0\\
        \dvrC{\calA} & \dvrC{\calB} & 0\\
        \dvrC{\calC} & \dvrC{\calD} & 0\\
        0 & I & 0\\
        0 & 0 & I
      \end{pmatrix} \dvrMax{\Sigma} \Bigg) ,
  \end{align*}
  where $\dvrMax{\Sigma}$ is the Lagrange multiplier for the constraint \eqref{eq:performanceAnalysisCorB}. By rearranging terms and partitioning $\dvrMax{\Sigma}$ as in \eqref{eq:SigmaPartitioning}, we rewrite the Lagrangian dual problem as
  \begin{align*}
    \sup_{\dvrMax{\Sigma} \succeq 0} & \inf_{\dvr{\lambda} \geq 0, \dvr{Q},\dvr{\calP}} ~ \tr\left(
      \begin{pmatrix}
        \dvrC{\calC} & \dvrC{\calD} & 0
      \end{pmatrix}
      \dvrMax{\Sigma}
      \begin{pmatrix}
        \dvrC{\calC} & \dvrC{\calD} & 0
      \end{pmatrix}^\top
      \right)\\
      & +
      \tr( \dvr{\calP}
      \left(\begin{pmatrix}
        \dvrC{\calA} & \dvrC{\calB} & 0
      \end{pmatrix}
      \dvrMax{\Sigma}
      \begin{pmatrix}
        \dvrC{\calA} & \dvrC{\calB} & 0
      \end{pmatrix}^\top - \dvrMax{\Sigma_{\chi\chi}} \right) )
      \\
      &+
      \dvr{\lambda} \gamma^2
      -
      \dvr{\lambda}
      \tr\left(
      \begin{pmatrix}
        0 & I & 0\\
        0 & 0 & I
      \end{pmatrix}
      \dvrMax{\Sigma}
      \begin{pmatrix}
        0 & I & 0\\
        0 & 0 & I
      \end{pmatrix}^\top
      \begin{pmatrix}
        I & -I\\
        -I & I
      \end{pmatrix}
      \right)
      \\
      & + \tr\left( \dvr{Q} \left(\Sigma_\text{nom} -\begin{pmatrix}
        0 & 0 & I
      \end{pmatrix}
      \dvrMax{\Sigma}
      \begin{pmatrix}
        0 & 0 & I
      \end{pmatrix}^\top\right) \right).
  \end{align*}
  Here, we see that if any of the constraints \eqref{eq:momentRelaxedProblemB}, \eqref{eq:momentRelaxedProblemC}, or \eqref{eq:momentRelaxedProblemE} are violated, then $\dvr{\calP}$, $\dvr{\lambda}$, or $\dvr{Q}$ can be chosen such that the objective value is unbounded from below. On the other hand, if \eqref{eq:momentRelaxedProblemB}, \eqref{eq:momentRelaxedProblemC}, and \eqref{eq:momentRelaxedProblemE} hold, then optimal choices for $\dvr{\calP}$, $\dvr{\lambda}$, and $\dvr{Q}$ are given by zero.
  
  Consequently, \eqref{eq:momentRelaxedProblem} is equivalent to the dual of \eqref{eq:performanceAnalysisCor}. By weak duality, the optimal value of \eqref{eq:momentRelaxedProblem} is smaller than or equal to the optimal value of \eqref{eq:performanceAnalysisCor}.

  \noindent \textbf{Step 2:} Strong duality and feasibility.

  To show that strong duality holds, we verify Slater's condition by proving that \eqref{eq:performanceAnalysisCor} and $\dvr{\lambda} \geq 0$ are strictly feasible.

  It helps to rewrite the constraint \eqref{eq:performanceAnalysisCorB} as
  \begin{align*}
    \begin{pmatrix}
      \dvrC{\calA^\top}\dvr{\calP} \dvrC{\calA} - \dvr{\calP} + \dvrC{\calC^\top} \dvrC{\calC} & \dvrC{\calA^\top}\dvr{\calP} \dvrC{\calB} + \dvrC{\calC^\top} \dvrC{\calD} & 0\\
      \dvrC{\calB^\top}\dvr{\calP} \dvrC{\calA} + \dvrC{\calD^\top} \dvrC{\calC} & \dvrC{\calB^\top}\dvr{\calP} \dvrC{\calB} + \dvrC{\calD^\top} \dvrC{\calD} - \dvr{\lambda}I & \dvr{\lambda}I\\
      0 & \dvr{\lambda} I & -\dvr{\lambda}I - \dvr{Q}
    \end{pmatrix} \prec 0.
  \end{align*}
  Then, we choose $\dvr{\calP}$ as the solution to the Lyapunov equation
  \begin{align*}
    \dvrC{\calA^\top} \dvr{\calP} \dvrC{\calA} - \dvr{\calP} + \dvrC{\calC^\top} \dvrC{\calC} = -I.
  \end{align*}
  Subsequently pick $\dvr{\lambda}$ large enough such that the left upper $2\times 2$ block of \eqref{eq:performanceAnalysisCorB} above is negative definite, and finally make $\dvr{Q}$ large enough such that the whole matrix is negative definite.
\end{proof}

\subsection{Robust performance analysis for independent disturbances}
\label{subsec:robustPerformanceAnalysisInd}

In this section, we derive the dual problem of \eqref{eq:momentRelaxedProblem} for the case of independent disturbances. As in Section \ref{subsec:robustPerformanceAnalysisCor}, this dual takes the form of a robust performance analysis problem typically found in dissipativity theory.

\begin{theorem}
  \label{thm:performanceAnalysisInd}
  Let $\dvrC{\calA}$ be stable. 
  Consider the case $\scrS = \scrS_\text{ind}$. In this case, \eqref{eq:momentRelaxedProblem} is equivalent to the optimization problem
  \begin{subequations}
    \label{eq:performanceAnalysisInd}
  \begin{align}
    &\minimize_{\dvr{\lambda} \geq 0,\dvr{\calP}, \dvr{Q}} ~~ \tr\left( \dvr{Q} \Sigma_\text{nom} \right) + \dvr{\lambda}\gamma^2 \label{eq:performanceAnalysisIndA}\\
    &\mathrm{s.t.} ~~ 
    \dvrC{\calA^\top} \dvr{\calP} \dvrC{\calA} - \dvr{\calP} + \dvrC{\calC^\top} \dvrC{\calC} \prec 0,\label{eq:performanceAnalysisIndB}\\
    &
    \begin{pmatrix}
      \dvrC{\calB} & 0\\
      \dvrC{\calD} & 0\\
      I & 0\\
      0 & I
    \end{pmatrix}^\top 
    \begin{pmatrix}
      \dvr{\calP} & 0 & 0 & 0\\
      0 & I & 0 & 0\\
      0 & 0 & -\dvr{\lambda}I & \dvr{\lambda}I\\
      0 & 0 & \dvr{\lambda}I & -\dvr{Q} - \dvr{\lambda}I
    \end{pmatrix}
    \begin{pmatrix}
      \dvrC{\calB} & 0\\
      \dvrC{\calD} & 0\\
      I & 0\\
      0 & I
    \end{pmatrix} \prec 0.\label{eq:performanceAnalysisIndC}
  \end{align}
\end{subequations}
The stability of $\dvrC{\calA}$ further guarantees the feasibility of \eqref{eq:performanceAnalysisInd}.
\end{theorem}

\begin{proof}
  By confining $\dvrMax{\Sigma}$ to the set $\scrS_\text{ind}$, we can rewrite the dual optimization problem corresponding to \eqref{eq:performanceAnalysisInd} as
  \begin{align*}
    &\sup_{\dvrMax{\Sigma} \succeq 0,\dvrMax{\Sigma}\in \scrS_\text{ind}} \inf_{\dvr{\lambda} \geq 0, \dvr{Q},\dvr{\calP}} ~\tr\left( \dvr{Q} \Sigma_\text{nom} \right) + \dvr{\lambda}\gamma^2 +\\
    &
    \tr\Bigg( \underbrace{
      % \begin{pmatrix}
      %   I & 0 & 0\\
      %   \dvrC{\calA} & \dvrC{\calB} & 0\\
      %   \dvrC{\calC} & \dvrC{\calD} & 0\\
      %   0 & I & 0\\
      %   0 & 0 & I
      % \end{pmatrix}^\top 
      \begin{pmatrix}
        \bullet
      \end{pmatrix}^\top
      \begin{pmatrix}
        -\dvr{\calP} & 0 & 0 & 0 & 0\\
        0 & \dvr{\calP} & 0 & 0 & 0\\
        0 & 0 & I & 0 & 0\\
        0 & 0 & 0 & -\dvr{\lambda}I & \dvr{\lambda}I\\
        0 & 0 & 0 & \dvr{\lambda}I & -\dvr{Q} - \dvr{\lambda}I
      \end{pmatrix}
      \begin{pmatrix}
        I & 0 & 0\\
        \dvrC{\calA} & \dvrC{\calB} & 0\\
        \dvrC{\calC} & \dvrC{\calD} & 0\\
        0 & I & 0\\
        0 & 0 & I
      \end{pmatrix}
    }_{=(\star)} \dvrMax{\Sigma} \Bigg) .
  \end{align*}
  Specifically, the left upper $1 \times 1$ block of $(\star)$ corresponds to \eqref{eq:performanceAnalysisIndB} and the right lower $2 \times 2$ block corresponds to \eqref{eq:performanceAnalysisIndC}. The remaining blocks of $(\star)$ are irrelevant, since the top right $1 \times 2$ and bottom left $2 \times 1$ blocks of $\dvrMax{\Sigma}$ are zero. The rest of the proof is analogous to the one of Theorem~\ref{thm:performanceAnalysisCor}.
\end{proof}

% \begin{remark}[Weaker assumptions]
%   The controllability assumption in Theorem~\ref{thm:performanceAnalysisCor} and Theorem~\ref{thm:performanceAnalysisInd} on $(\dvrC{\calA},\dvrC{\calB})$ can be avoided.

%   Indeed, if we add uncontrollable modes to $(\dvrC{\calA},\dvrC{\calB})$, then the corresponding rows and columns of $\dvrMax{\Sigma}$ must be zero due to \eqref{eq:momentRelaxedProblemB} and \eqref{eq:momentRelaxedProblemD}, i.e., the optimal value of \eqref{eq:momentRelaxedProblem} does not change. To establish the same for, e.g., \eqref{eq:performanceAnalysisInd}, candidate solutions $(\dvr{\calP},\dvr{Q},\dvr{\lambda})$ must be perturbed in such a way that \eqref{eq:performanceAnalysisIndB} is strictly satisfied and the cost is only minimally affected. Subsequently, uncontrollable modes can be added to \eqref{eq:}. Due to the stability assumptions on $\dvrC{\calA}$, these modes are stable and, thus, we can simply choose the corresponding block of $\dvr{\calP}$ sufficiently large that \eqref{eq:performanceAnalysisIndB} is satisfied. Entries of $\dvr{\calP}$ belonging to cross-terms of uncontrollable modes and controllable modes can be set to zero, since the uncontrollable modes live in an invariant subspace of $\dvrC{\calA}$. Note that \eqref{eq:performanceAnalysisIndC} is not affected by this choice, since the altered block of $\dvr{\calP}$ belongs to the co-kernel of $\dvrC{\calB}$.
% \end{remark}

\section{Robust synthesis}
\label{sec:robustSynthesis}

In this section, we build on the results of the previous sections to also optimize over the controller parameters $\dvrC{A_c}$, $\dvrC{B_c}$, $\dvrC{C_c}$, and $\dvrC{D_c}$. Unfortunately, the matrix inequality constraints in Theorem~\ref{thm:performanceAnalysisCor} and Theorem~\ref{thm:performanceAnalysisInd} are not simultaneously linear in the controller parameters and the performance certificates $\dvr{\calP}$, $\dvr{Q}$, and $\dvr{\lambda}$.

A first step to convexify \eqref{eq:performanceAnalysisCor} and \eqref{eq:performanceAnalysisInd} in the controller parameters is the application of the Schur complement.

\begin{proposition}[Preliminary synthesis problem 1]
  \label{prop:convexSynthesisCor}
  Let $\dvrC{\calA}$ be stable. Then, the optimization problem \eqref{eq:performanceAnalysisCor} in Theorem~\ref{thm:performanceAnalysisCor} is equivalent to
  \begin{align*}
    &\minimize_{\dvr{\lambda} \geq 0, \dvr{Q},\dvr{\calP} \succ 0} ~~ \tr\left( \dvr{Q} \Sigma_\text{nom} \right) + \dvr{\lambda}\gamma^2\\
    &\mathrm{s.t.} ~~ 
    \begin{pmatrix}
      -\dvr{\calP} & 0 & 0 & \dvrC{\calA^\top}\dvr{\calP} & \dvrC{\calC^\top}\\
      0 & -\dvr{\lambda}I & \dvr{\lambda}I & \dvrC{\calB^\top}\dvr{\calP} & \dvrC{\calD^\top}\\
      0 & \dvr{\lambda}I & -\dvr{\lambda}I-\dvr{Q} & 0 & 0\\
      \dvr{\calP}\dvrC{\calA} & \dvr{\calP}\dvrC{\calB} & 0 & -\dvr{\calP} & 0\\
      \dvrC{\calC} & \dvrC{\calD} & 0 & 0 & -I
    \end{pmatrix}\prec 0.
  \end{align*}
\end{proposition}

\begin{proof}
  First, we notice that the left upper $1 \times 1$ block of \eqref{eq:performanceAnalysisCorB} corresponds to the Lyapunov inequality
  \begin{align*}
    \dvrC{\calA^\top} \dvr{\calP} \dvrC{\calA} - \dvr{\calP} + \dvrC{\calC^\top} \dvrC{\calC} \prec 0.
  \end{align*}
  Consequently, since $\dvrC{\calA}$ is stable, it poses no restriction to require $\dvr{\calP} \succ 0$.

  Next, denote the matrix inequality constraint from Theorem~\ref{thm:performanceAnalysisCor} as
  \begin{align*}
    \resizebox{\linewidth}{!}{$
    \begin{pmatrix}
      -\dvr{\calP} & 0 & 0\\
      0 & -\dvr{\lambda}I & \dvr{\lambda}I\\
      0 & \dvr{\lambda}I & -\dvr{Q} - \dvr{\lambda}I
    \end{pmatrix}
    +
    \begin{pmatrix}
      \dvrC{\calA^\top} & \dvrC{\calC^\top}\\
      \dvrC{\calB^\top} & \dvrC{\calD^\top}\\
      0 & 0
    \end{pmatrix}
    \begin{pmatrix}
      \dvr{\calP} & 0\\
      0 & I
    \end{pmatrix}
    \begin{pmatrix}
      \dvrC{\calA^\top} & \dvrC{\calC^\top}\\
      \dvrC{\calB^\top} & \dvrC{\calD^\top}\\
      0 & 0
    \end{pmatrix}^\top \prec 0.$}
  \end{align*}

  In this form, we see that, since $\diag(\dvr{\calP},I)$ is positive definite, the Schur complement can be applied. This yields the equivalent matrix inequality
  \begin{align*}
    \begin{pmatrix}
      -\dvr{\calP} & 0 & 0 & \dvrC{\calA^\top} & \dvrC{\calC^\top}\\
      0 & -\dvr{\lambda}I & \dvr{\lambda}I & \dvrC{\calB^\top} & \dvrC{\calD^\top}\\
      0 & \dvr{\lambda}I & -\dvr{Q} - \dvr{\lambda}I & 0 & 0\\
      \dvrC{\calA} & \dvrC{\calB} & 0 & -\dvr{\calP^{-1}} & 0\\
      \dvrC{\calC} & \dvrC{\calD} & 0 & 0 & -I
    \end{pmatrix} \prec 0.
  \end{align*}
  This matrix inequality is nonlinear in the certificate $\dvr{\calP}$. Thus, we multiply this matrix inequality from both sides with $\diag(I,I,I,\dvr{\calP},I)$ to achieve our desired result.
\end{proof}

\begin{proposition}[Preliminary synthesis problem 2]
  \label{prop:convexSynthesisInd}
  Let $\dvrC{\calA}$ be stable. Then, the optimization problem from Theorem~\ref{thm:performanceAnalysisInd} is equivalent to
  \begin{subequations}
  \begin{align}
    &\minimize_{\dvr{\lambda} \geq 0, \dvr{\calP} \succ 0, \dvr{Q}} ~~ \tr\left( \dvr{Q} \Sigma_\text{nom} \right) + \dvr{\lambda}\gamma^2
    \label{}\\
    &\mathrm{s.t.} ~~ 
    \begin{pmatrix}
      -\dvr{\calP} & \dvrC{\calA^\top}\dvr{\calP} & \dvrC{\calC^\top}\\
      \dvr{\calP}\dvrC{\calA} & -\dvr{\calP} & 0\\
      \dvrC{\calC} & 0 & -I
    \end{pmatrix} \prec 0,\\
    &
    \begin{pmatrix}
      -\dvr{\lambda}I & \dvr{\lambda}I & \dvrC{\calB^\top}\dvr{\calP} & \dvrC{\calD^\top}\\
      \dvr{\lambda}I & -\dvr{Q} - \dvr{\lambda}I & 0 & 0\\
      \dvr{\calP}\dvrC{\calB} & 0 & -\dvr{\calP} & 0\\
      \dvrC{\calD} & 0 & 0 & -I
    \end{pmatrix}\prec 0.
  \end{align}
\end{subequations}
\end{proposition}
\begin{proof}
  The proof is analogous to the proof of Proposition~\ref{prop:convexSynthesisCor}.
\end{proof}

Unfortunately, the usage of the Schur complement alone is not enough to convexify the synthesis problem. Proposition~\ref{prop:convexSynthesisCor} and Proposition~\ref{prop:convexSynthesisInd} give proof for that, as terms of the form $\dvr{\calP} \dvrC{\calA}$, $\dvr{\calP} \dvrC{\calB}$ appear. These terms become bilinear, when we add the controller parameters to the decision variables. Fortunately, there is yet another celebrated result from control theory that can be invoked to convexify the synthesis problem.

\begin{theorem}[Version of \cite{599969}, cmp. \cite{masubuchi1998lmi}]
  \label{thm:schererParametrization}
  Consider the definitions of the matrices $\dvr{\bP}$, $\dvrC{\bA}$, $\dvrC{\bB}$, $\dvrC{\bC}$, and $\dvrC{\bD}$ given by
  \begin{align*}
    &\dvr{\bP}
    =
    \begin{pmatrix}
      \dvr{Y} & I\\
      I & \dvr{X}
    \end{pmatrix}, \quad
    \left(
    \begin{array}{c|cc}
      \dvrC{\bA}  & \dvrC{\bB}\\
      \hline
      \dvrC{\bC}  & \dvrC{\bD} 
    \end{array}
    \right)
    =\\
    &=
    \left(
    \begin{array}{cc|cc}
      A\dvr{Y} + B_u\dvrC{M} & A + B_u\dvrC{N}C_y & B_w + B_u\dvrC{N}D_{yw}\\
      \dvrC{K} & \dvr{X}A + \dvrC{L}C_y & \dvr{X}B_w + \dvrC{L}D_{yw}\\
      \hline
      C_z\dvr{Y} + D_{zu}\dvrC{M} & C_z + D_{zu}\dvrC{N}C_y & D_{zw} + D_{zu}\dvrC{N}D_{yw}
    \end{array}
    \right)
  \end{align*}
  in terms of the parameters $\dvr{X} = \dvr{X^\top} \in \bbR^{n_x \times n_x}$, $\dvr{Y} = \dvr{Y^\top} \in \bbR^{n_x \times n_x}$, $\dvrC{K} \in \bbR^{n_x \times n_x}$, $\dvrC{L} \in \bbR^{n_x \times n_y}$, $\dvrC{M} \in \bbR^{n_u \times n_x}$, $\dvrC{N} \in \bbR^{n_u \times n_y}$. These matrices shall be related to the controller parameters $\dvrC{A_c}$, $\dvrC{B_c}$, $\dvrC{C_c}$, and $\dvrC{D_c}$ and the multiplier matrix $\dvr{\calP}$ through the equations
  \begin{align}
    \begin{pmatrix}
      \calT^\top \dvr{\calP} \dvrC{\calA}\calT & \calT^\top \dvr{\calP}\dvrC{\calB}\\
      \dvrC{\calC}\calT & \dvrC{\calD}
    \end{pmatrix}
    &=
    \begin{pmatrix}
      \dvrC{\bA}  & \dvrC{\bB}\\
      \dvrC{\bC}  & \dvrC{\bD} 
    \end{pmatrix},
    & \calT^\top \dvr{\calP} \calT &= \dvr{\bP} . \label{eq:schererParametrization}
  \end{align}
  The following facts are true:
  \begin{itemize}
    \item For any $\dvr{\calP} \succ 0$, $\dvrC{A_c}$, $\dvrC{B_c}$, $\dvrC{C_c}$, and $\dvrC{D_c}$, there exist a congruence transformation via nonsingular matrix $\calT$ and parameters $(\dvr{X},\dvr{Y},\dvrC{K},\dvrC{L},\dvrC{M},\dvrC{N})$ such that \eqref{eq:schererParametrization} holds.
    \item For any parameters $(\dvr{X},\dvr{Y},\dvrC{K},\dvrC{L},\dvrC{M},\dvrC{N})$ with $\dvr{\bP} \succ 0$, there exist a congruence transformation via nonsingular matrix $\calT$, a matrix $\dvr{\calP}$, and controller parameters $\dvrC{A_c}$, $\dvrC{B_c}$, $\dvrC{C_c}$, and $\dvrC{D_c}$ such that \eqref{eq:schererParametrization} holds.
  \end{itemize}
\end{theorem}

With the aid of Theorem~\ref{thm:schererParametrization}, we can easily extend the results of Proposition~\ref{prop:convexSynthesisCor} and Proposition~\ref{prop:convexSynthesisInd} to convex controller synthesis.

\begin{theorem}[Robust synthesis 1]
  \label{thm:convexSynthesisCor}
  The optimal value of the optimization problem \eqref{eq:mainProblem} for $\scrW = \scrW_\text{cor}$ is equal to the optimal value of the convex program
  \begin{align*}
    &\minimize_{\dvr{\lambda} \geq 0, \dvr{Q},\dvr{X},\dvr{Y},\dvrC{K},\dvrC{L},\dvrC{M},\dvrC{N}} ~~ \tr\left( \dvr{Q} \Sigma_\text{nom} \right) + \dvr{\lambda}\gamma^2\\
    &\mathrm{s.t.} ~~ 
    \begin{pmatrix}
      -\dvr{\bP} & 0 & 0 & \dvrC{\bA^\top} & \dvrC{\bC^\top}\\
      0 & -\dvr{\lambda}I & \dvr{\lambda}I & \dvrC{\bB^\top} & \dvrC{\bD^\top}\\
      0 & \dvr{\lambda}I & -\dvr{Q} - \dvr{\lambda}I & 0 & 0\\
      \dvrC{\bA} & \dvrC{\bB} & 0 & -\dvr{\bP} & 0\\
      \dvrC{\bC} & \dvrC{\bD} & 0 & 0 & -I
    \end{pmatrix}\prec 0,
    ~~ \dvr{\bP} \succ 0.
  \end{align*}
\end{theorem}

\begin{proof}
  By Theorems~\ref{thm:finiteHorizonProblemEquivalent}, \ref{thm:exactnessMomentRelaxation}, and \ref{thm:performanceAnalysisCor}, the optimal value of \eqref{eq:mainProblem} for $\scrW = \scrW_\text{cor}$ equals the optimal value of the optimization problem in Proposition~\ref{prop:convexSynthesisCor} when we additionally minimize over the controller parameters $\dvrC{A_c}$, $\dvrC{B_c}$, $\dvrC{C_c}$, and $\dvrC{D_c}$.

  To establish the equivalence with our convex formulation, consider any feasible point $(\dvrC{A_c}, \dvrC{B_c}, \dvrC{C_c}, \dvrC{D_c}, \dvr{\lambda}, \dvr{Q}, \dvr{\calP})$ of the extended optimization problem from Proposition~\ref{prop:convexSynthesisCor}. By the first part of Theorem~\ref{thm:schererParametrization}, there exists a nonsingular congruence transformation matrix $\calT$ and parameters $(\dvr{X}, \dvr{Y}, \dvrC{K}, \dvrC{L}, \dvrC{M}, \dvrC{N})$ such that the relations \eqref{eq:schererParametrization} hold with $\dvr{\bP} = \calT^\top \dvr{\calP} \calT \succ 0$.
  
  Applying the congruence transformation $\diag(\calT, I, I, \calT, I)$ to the matrix inequality constraint in Proposition~\ref{prop:convexSynthesisCor} and substituting the parameterized variables $(\dvrC{\bA}, \dvrC{\bB}, \dvrC{\bC}, \dvrC{\bD}, \dvr{\bP})$ defined in Theorem~\ref{thm:schererParametrization}, we obtain a feasible point for our convex program with the same objective value.
  
  Conversely, by the second part of Theorem~\ref{thm:schererParametrization}, any feasible point of our convex program can be transformed back to a feasible point  $(\dvrC{A_c}, \dvrC{B_c}, \dvrC{C_c}, \dvrC{D_c}, \dvr{\lambda}, \dvr{Q}, \dvr{\calP})$ of the extended optimization problem from Proposition~\ref{prop:convexSynthesisCor} with the same objective value. This establishes the equivalence of optimal values. Moreover, the condition $\dvr{\calP}\succ 0$, which is guaranteed by $\dvr{\bP} \succ 0$, ensures that $\dvrC{\calA}$ is stable.
\end{proof}

The same convexification can be carried out with Proposition~\ref{prop:convexSynthesisInd}.

\begin{theorem}[Robust synthesis 2]
  \label{thm:convexSynthesisInd}
  The optimal value of the optimization problem \eqref{eq:mainProblem} for $\scrW = \scrW_\text{ind}$ is equal to the optimal value of the convex program
  \begin{align*}
    &\minimize_{\dvr{\lambda} \geq 0, \dvr{Q},\dvr{X},\dvr{Y},\dvrC{K},\dvrC{L},\dvrC{M},\dvrC{N}} ~~ \tr\left( \dvr{Q} \Sigma_\text{nom} \right) + \dvr{\lambda}\gamma^2\\
    &\begin{aligned}\mathrm{s.t.} ~~ 
    &\begin{pmatrix}
      -\dvr{\bP} & \dvrC{\bA^\top} & \dvrC{\bC^\top}\\
      \dvrC{\bA} & -\dvr{\bP} & 0\\
      \dvrC{\bC} & 0 & -I
    \end{pmatrix} \prec 0, ~~ \dvr{\bP} \succ 0,\\
    &
    \begin{pmatrix}
      -\dvr{\lambda}I & \dvr{\lambda}I & \dvrC{\bB^\top} & \dvrC{\bD^\top}\\
      \dvr{\lambda}I & -\dvr{Q} - \dvr{\lambda}I & 0 & 0\\
      \dvrC{\bB} & 0 & -\dvr{\bP} & 0\\
      \dvrC{\bD} & 0 & 0 & -I
    \end{pmatrix}\prec 0.
    \end{aligned}
  \end{align*}
\end{theorem}

\begin{proof}
This result follows from carrying out the same steps as in the proof of Theorem~\ref{thm:convexSynthesisCor}.
\end{proof}

\section{Relation to robust $H_2$ synthesis}

The problem studied in this article is closely related to a specific robust $H_2$ synthesis problem. Indeed, we show in this section that the robust analysis problem \eqref{eq:performanceAnalysisInd} can be interpreted as robust $H_2$ analysis for an uncertain system interconnection depicted in Figure \ref{fig:H2Circuit}.

\begin{figure}
  \centering
  \input{stateSpaceDeltaStructureController.tex}
  \caption{This figure shows a closed-loop system described by matrices $\dvrC{\calA},\dvrC{\calB},\dvrC{\calC},\dvrC{\calD}$. The system has an uncertain real block $\Delta$ in the performance channel from $\hat{w}$ to $z$. The symbol $Z_-$ denotes the back shift operator.}
  \label{fig:H2Circuit}
\end{figure}
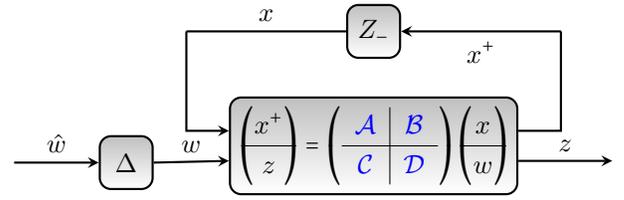

This uncertain system interconnection is given by a linear (closed-loop) system with a real block $\Delta$ in the input of the performance channel, i.e., we study the uncertain system
\begin{align}
  \begin{pmatrix}
    \chi_{k+1}\\
    z_{k+1}
  \end{pmatrix}
  &=
  \begin{pmatrix}
    \dvrC{\calA} & \dvrC{\calB}\Delta\\
    \dvrC{\calC} & \dvrC{\calD}\Delta
  \end{pmatrix}
  \begin{pmatrix}
    \chi_k\\
    \hat{w}_k
  \end{pmatrix}
  =
  \begin{pmatrix}
    \dvrC{\calA} & \dvrC{\calB_\Delta}\\
    \dvrC{\calC} & \dvrC{\calD_\Delta}
  \end{pmatrix}
  \begin{pmatrix}
    \chi_k\\
    \hat{w}_k
  \end{pmatrix}. \label{eq:uncertainSystem}
\end{align}
Here, $\dvrC{\calB_\Delta}$ and $\dvrC{\calD_\Delta}$ denote \emph{perturbed} versions of the matrices $\dvrC{\calB}$ and $\dvrC{\calD}$, respectively. The goal is to analyze the $\Sigma_\text{nom}$-weighted $H_2$ norm of the transfer function from $\hat{w}$ to $z$, which is the closed-loop transfer function of a system with a controller. With $\Sigma_\text{nom}$-weighted, we mean that the input matrices $\dvrC{\calB_\Delta}$ and $\dvrC{\calD_\Delta}$ are scaled with the static weight $\Sigma_\text{nom}^{\frac{1}{2}}$. In this course, we invoke the following result from robust control theory \cite{scherer2000linear}.

\begin{theorem}[Bounding the $H_2$ norm \cite{scherer2000linear}]
  \label{thm:H2}
  Let $\dvrC{\calA}$ be stable. The ($\Sigma_\text{nom}$-weighted) $H_2$ norm of \eqref{eq:uncertainSystem} is smaller than $\rho$ if and only if there exists a matrix $\dvr{\calP} \succ 0$ such that
  \begin{subequations}
    \label{eq:H2normConditions}
  \begin{align}
    \dvrC{\calA^\top} \dvr{\calP} \dvrC{\calA} - \dvr{\calP} + \dvrC{\calC^\top} \dvrC{\calC} \prec 0,\\
    \tr( (\dvrC{\calB_\Delta^\top}\dvr{\calP}\dvrC{\calB_\Delta} + \dvrC{\calD_\Delta^\top} \dvrC{\calD_\Delta})\Sigma_\text{nom} ) < \rho^2.
  \end{align}
  \end{subequations}
\end{theorem}

We assume that $\Delta$ is fixed but unknown except for the fact that it satisfies $\|(I - \Delta)\Sigma_\text{nom}^{\frac{1}{2}}\|_F \leq \gamma$, i.e., it is $\gamma$-close to the identity matrix in the weighted Frobenius norm. This constraint can alternatively be written as
\begin{align*}
  \tr\left(
    \begin{pmatrix}
      \Delta\\
      I
    \end{pmatrix}^\top
    \begin{pmatrix}
      I & -I\\
      -I & I
    \end{pmatrix}
    \begin{pmatrix}
      \Delta\\
      I
    \end{pmatrix}
    \Sigma_\text{nom}
  \right)
  \leq \gamma^2.
\end{align*}
In the next step, we add this constraint to the $H_2$ norm condition and arrive at the alternative result that the weighted $H_2$ norm of \eqref{eq:uncertainSystem} is smaller than $\rho$ if and only if there exists a matrix $\dvr{\calP} \succ 0$ and a scalar $\dvr{\lambda} \geq 0$ such that
\begin{align*}
  0 & \succ \dvrC{\calA^\top} \dvr{\calP} \dvrC{\calA} - \dvr{\calP} + \dvrC{\calC^\top} \dvrC{\calC},\\
  \rho^2 &>\tr( (\Delta^\top\dvrC{\calB^\top}\dvr{\calP}\dvrC{\calB}\Delta + \Delta\dvrC{\calD^\top} \dvrC{\calD}\Delta) \Sigma_\text{nom} )\\
  &- \dvr{\lambda}\tr\left(
    \begin{pmatrix}
      \Delta\\
      I
    \end{pmatrix}^\top
    \begin{pmatrix}
      I & -I\\
      -I & I
    \end{pmatrix}
    \begin{pmatrix}
      \Delta\\
      I
    \end{pmatrix}
    \Sigma_\text{nom}
  \right) + \dvr{\lambda}\gamma^2.
\end{align*}
The equivalence follows directly from Theorem~\ref{thm:H2}, since we added a non-negative term to the inequality, which is zero for $\dvr{\lambda} = 0$. In the next step, we introduce an auxiliary variable $\dvr{Q}$ and rewrite the condition as
\begin{align*}
  \Delta^\top(\dvrC{\calB^\top}\dvr{\calP}\dvrC{\calB} + \dvrC{\calD^\top} \dvrC{\calD})\Delta + \dvr{\lambda} \begin{pmatrix}
    \Delta\\
    I
  \end{pmatrix}^\top
  \begin{pmatrix}
    -I & I\\
    I & -I
  \end{pmatrix}
  \begin{pmatrix}
    \Delta\\
    I
  \end{pmatrix} \prec \dvr{Q},
\end{align*}
and $\tr( \dvr{Q} \Sigma_\text{nom} ) + \dvr{\lambda} \gamma^2 < \rho^2$. Equivalently, we can express this as
\begin{align}
  \begin{pmatrix}
    \Delta\\
    I
  \end{pmatrix}^\top
  \underbrace{
  \begin{pmatrix}
    \dvrC{\calB} & 0\\
    \dvrC{\calD} & 0\\
    I & 0\\
    0 & I
  \end{pmatrix}^\top
  \left(
  \begin{array}{cccc}
    \dvr{\calP} & 0 & 0 & 0\\
    0 & I & 0 & 0\\
    0 & 0 & -\dvr{\lambda}I & \dvr{\lambda}I\\
    0 & 0 &  \dvr{\lambda}I & -\dvr{Q} -\dvr{\lambda}I
  \end{array}
  \right)
  \begin{pmatrix}
    \dvrC{\calB} & 0\\
    \dvrC{\calD} & 0\\
    I & 0\\
    0 & I
  \end{pmatrix}}_{(\star)}
  \begin{pmatrix}
    \Delta\\
    I
  \end{pmatrix}
  \label{eq:DeltaDependentMatrixInequality}
\end{align}
being negative definite, and $\tr( \dvr{Q} \Sigma_\text{nom} ) + \dvr{\lambda} \gamma^2 < \rho^2$. 

The reformulations so far preserved the exactness of the $H_2$ norm condition, but they are impractical, since they depend on the unknown $\Delta$. To obtain a more practical formulation, we can use the fact that \eqref{eq:DeltaDependentMatrixInequality} is negative definite if $(\star)$ is negative definite. The matrix $(\star)$ does not depend on $\Delta$, but it is not immediately clear whether moving from \eqref{eq:DeltaDependentMatrixInequality} to $(\star)$ preserves the exactness of the $H_2$ norm condition.

Then, we end up with the robust $H_2$ analysis problem
\begin{subequations}
  \label{eq:robustH2Analysis}
\begin{align}
  &\minimize_{\dvr{\lambda} \geq 0, \dvr{Q},\dvr{\calP}\succ 0} ~~ \tr\left( \dvr{Q} \Sigma_\text{nom} \right) + \dvr{\lambda}\gamma^2\\
  &\mathrm{s.t.} ~~ 
  \dvrC{\calA^\top} \dvr{\calP} \dvrC{\calA} - \dvr{\calP} + \dvrC{\calC^\top} \dvrC{\calC} \prec 0,\\
  &
  \begin{pmatrix}
    \dvrC{\calB} & 0\\
    \dvrC{\calD} & 0\\
    I & 0\\
    0 & I
  \end{pmatrix}^\top 
  \begin{pmatrix}
    \dvr{\calP} & 0 & 0 & 0\\
    0 & I & 0 & 0\\
    0 & 0 & -\dvr{\lambda}I & \dvr{\lambda}I\\
    0 & 0 & \dvr{\lambda}I & -\dvr{Q} - \dvr{\lambda}I
  \end{pmatrix}
  \begin{pmatrix}
    \dvrC{\calB} & 0\\
    \dvrC{\calD} & 0\\
    I & 0\\
    0 & I
  \end{pmatrix} \prec 0,
\end{align}
\end{subequations}
which is the same as the distributionally robust analysis problem \eqref{eq:performanceAnalysisInd}. Since \eqref{eq:robustH2Analysis} and \eqref{eq:performanceAnalysisInd} are identical, we do not need to convexify \eqref{eq:robustH2Analysis} in the controller parameters, which we have already done for \eqref{eq:performanceAnalysisInd} in Section \ref{sec:robustSynthesis}.

\begin{corollary}[Transport plans and exactness of \eqref{eq:robustH2Analysis}]
  \label{cor:exactnessRobustH2Analysis}
  Assume that $\dvrC{\calA}$ is stable and $\Sigma_\text{nom} \succ 0$. Then, the robust $H_2$ analysis problem \eqref{eq:robustH2Analysis} is exact in the sense that there exists a worst-case uncertainty $\Delta \in \bbR^{n_w \times n_w}$ with $\|(I-\Delta)\Sigma_\text{nom}^\frac{1}{2}\|_F \leq \gamma$ such that the squared $\Sigma_\text{nom}$-weighted $H_2$ norm of the system \eqref{eq:uncertainSystem} is equal to the optimal value of \eqref{eq:robustH2Analysis}. This worst-case uncertainty corresponds to the transport plan of the Wasserstein distance in 3) of Theorem~\ref{thm:gelbrichProperties} between $\bbP_\text{nom}$ and $\bbP_{w(t)}$.
\end{corollary}
\begin{proof}
  \textbf{Step 1:} Equivalence of optimization problems. 

  By direct comparison of constraints and objectives, the robust $H_2$ analysis problem \eqref{eq:robustH2Analysis} is identical to the distributionally robust analysis problem \eqref{eq:performanceAnalysisInd}. Therefore, by Theorem~\ref{thm:performanceAnalysisInd}, both problems are equivalent to the original DRC problem \eqref{eq:finiteHorizonProblem}.

  \textbf{Step 2:} Construction of a worst-case disturbance.

  By the strong duality result in Theorem~\ref{thm:performanceAnalysisInd}, an optimal $\dvrMax{\Sigma}$ to problem \eqref{eq:momentRelaxedProblem} exists (cmp. \cite{luenberger1997optimization}, page 217). Based on such an optimal $\dvrMax{\Sigma}$, the construction \eqref{eq:worstCasePolicy} in Proposition~\ref{prop:finiteHorizonLowerBound} yields a worst-case disturbance sequence $w \in \scrW_\text{ind}$ that is optimal for problem \eqref{eq:worstCaseProblemAlternative}. By the independence assumption defining $\scrW_\text{ind}$, the random variables $w(t)$ for $t \in \bbN_0$ are independent. Moreover, by the explicit construction \eqref{eq:worstCasePolicy}, the random variables $w(t)$ are i.i.d. and Gaussian. Specifically, the independence assumption manifests itself in $\dvrMax{\Sigma_{\chi w}} = 0$ and $\dvrMax{\Sigma_{\chi \hat{w}}} = 0$ as specified in \eqref{eq:SigmaPartitioning}. Consequently, the construction \eqref{eq:worstCasePolicy} yields random variables $w(t) = F_{wd} d(t)$, where $d(t) \sim \calN(0,I)$ is i.i.d..

  \textbf{Step 3:} Existence of transport plans.

  Since the Wasserstein constraint $\bbW(\bbP_{w(t)}, \bbP_\text{nom}) \leq \gamma$ holds and both distributions are Gaussian with zero means, with $\Sigma_\text{nom} \succ 0$ by assumption, item 3) of Theorem~\ref{thm:gelbrichProperties} implies the existence of random variables $w(t) \sim \bbP_{w(t)}$ and $\hat{w}(t) \sim \bbP_\text{nom}$ and an optimal linear transport plan $\Delta \in \bbR^{n_w \times n_w}$ such that $w(t) = \Delta \hat{w}(t)$. This transport plan satisfies $\|(I-\Delta)\Sigma_\text{nom}^{1/2}\|_F \leq \gamma$ and does not depend on $t$, since the marginal distributions of $w$ are i.i.d.. Consequently, there exists a $\Delta \in \bbR^{n_w \times n_w}$ such that the squared $\Sigma_\text{nom}$-weighted $H_2$ norm of the system \eqref{eq:uncertainSystem}, which equals the objective value of problem \eqref{eq:worstCaseProblemAlternative} under the optimal disturbance sequence $w=\{\Delta \hat{w}(0), \Delta \hat{w}(1), \ldots\}$, achieves the optimal value of \eqref{eq:robustH2Analysis}.
  %
  % \textbf{Step 4:} Compactness and continuity arguments. \sy{Comment: would also need to assume $\Sigma_{\rm nom}\succ 0$ for the compactness}
  %
  % Since $\Sigma_\text{nom} \succ 0$ by assumption, the constraint set $\calK := \{\Delta \in \bbR^{n_w \times n_w} : \|(I-\Delta)\Sigma_\text{nom}^{1/2}\|_F \leq \gamma\}$ is compact, being a closed and bounded subset of $\bbR^{n_w \times n_w}$. The $\Sigma_\text{nom}$-weighted $H_2$ norm defined by \eqref{eq:H2normConditions} is continuous in $\Delta$ since it involves only matrix operations and the trace functional. 
  %
  % \textbf{Step 5:} Attainment of the supremum.
  %
  % By the extreme value theorem, the continuous function $\Delta \mapsto \text{squared $\Sigma_\text{nom}$-weighted $H_2$ norm}$ attains its supremum over the compact set $\mathcal{K}$. Therefore, there exists $\Delta^* \in \mathcal{K}$ such that the squared $\Sigma_\text{nom}$-weighted $H_2$ norm equals the optimal value of \eqref{eq:robustH2Analysis}. This optimal transport plan $\Delta^*$ corresponds to a worst-case disturbance for the original problem \eqref{eq:finiteHorizonProblem}, establishing the exactness of the robust $H_2$ analysis problem.
\end{proof}

It is worth noting that Corollary \ref{cor:exactnessRobustH2Analysis} does not only establish the equivalence of the distributionally robust analysis problem \eqref{eq:performanceAnalysisInd} and the robust $H_2$ analysis problem \eqref{eq:robustH2Analysis}, but also shows that the latter is exact in the sense that there exists a worst-case disturbance that achieves the supremum of the $\Sigma_\text{nom}$-weighted $H_2$ norm.

\section{Simulation Results}
\label{sec:simulations}

To validate the proposed distributionally robust controller synthesis framework, we conduct simulations on a linearized wind turbine blade control system, inspired by \cite{van2015distributionally}, modeled as a discrete-time linear time-invariant (LTI) system. The system state is \( x = (\omega, h, \dot{h}, \phi, \dot{\phi}, \beta, \dot{\beta})^\top\), where \( \omega \) is the rotor speed (rad/s), \( h \) is the flapwise displacement (m), \( \phi \) is the torsional angle (rad), and \( \beta \) is the pitch angle (rad). The control input is \( u \), representing the pitch rate command (rad/s), and the disturbance input is \( w = (v_x, v_z)^\top \), representing wind speed components (m/s) in the axial and vertical directions, respectively. The performance outputs are \( z = (\omega, h, \phi, u)^\top \), and the measured outputs are \( y = (\omega, h, \phi)^\top \). The control objective is to minimize the root-mean-square (RMS) values of flapwise displacement (\( h \), target RMS \(\leq 2\) m) and torsional angle (\( \phi \), target RMS \(\leq 0.02\) rad) under stochastic wind turbulence.

The system dynamics are described by:
\begin{subequations}
	\label{eq:expSystem}
	\begin{align}
		x(t+1) &= A x(t) + B_w w(t) + B_u u(t), \label{eq:expSystemA} \\
		z(t) &= C_z x(t) + D_{zw} w(t) + D_{zu} u(t), \label{eq:expSystemB} \\
		y(t) &= C_y x(t) + D_{yw} w(t), \label{eq:expSystemC}
	\end{align}
\end{subequations}
where \( A \in \mathbb{R}^{7 \times 7} \), \( B_w \in \mathbb{R}^{7 \times 2} \), \( B_u \in \mathbb{R}^{7 \times 1} \), \( C_z \in \mathbb{R}^{4 \times 7} \), \( D_{zw} \in \mathbb{R}^{4 \times 2} \), \( D_{zu} \in \mathbb{R}^{4 \times 1} \), \( C_y \in \mathbb{R}^{4 \times 7} \), and \( D_{yw} \in \mathbb{R}^{4 \times 2} \) are the system matrices, resulting from a continuous-time model which is discretized with a sampling time of \( \Delta t = 0.05 \) s. The disturbance \( w(t) \) is modeled as a zero-mean random process with nominal covariance \( \Sigma_\text{nom} = I_2 \). Details, such as the exact specification of the system matrices, are provided in the accompanying repository~\url{https://github.com/SphinxDG/Distri-\\butionallyrobustLMIsynthesisforLTIsystems}.

We use the correlated ambiguity set \( \scrW_\text{cor} \) instead of the independent set \( \scrW_\text{ind} \) because wind turbine blades are sensitive to temporally correlated wind disturbances. Wind turbulence exhibits low-frequency components that persist over time, causing sustained loads on the blade. The ambiguity set \( \scrW_\text{cor} \) allows disturbances to capture these temporal dependencies, enabling the controller to mitigate worst-case scenarios where disturbances amplify specific modes, such as flapwise or torsional oscillations.

The controller synthesis is based on the problem \eqref{eq:mainProblem} with the correlated ambiguity set \( \scrW_\text{cor} \). The nominal covariance matrix is \( \Sigma_\text{nom} = I_2 \), and the Wasserstein radius is set to \( \gamma = 0.5 \).

We synthesized two controllers: a standard $H_2$ controller, minimizing the nominal $H_2$ cost, and a DR controller, solving \eqref{eq:mainProblem} for \( \scrW_\text{cor} \). The DR controller is designed using the convex LMI formulation in Theorem~\ref{thm:convexSynthesisCor}. The $H_2$ controller is synthesized using standard LMI-based $H_2$ synthesis. Both controllers are implemented as output-feedback controllers of the form \eqref{eq:controller}, with \( \dvrC{A_c}, \dvrC{B_c}, \dvrC{C_c}, \dvrC{D_c} \) derived from the LMI solutions via the parametrization in Theorem~\ref{thm:schererParametrization}.

To evaluate robustness, we compute worst-case output variances for each performance output (\( \omega, h, \phi, u \)) individually under the ambiguity set \( \scrW_\text{cor} \) as follows. For each output, a worst-case disturbance is designed by solving the moment relaxation problem \eqref{eq:momentRelaxedProblem}, which maximizes the output variance subject to a discrete-time Lyapunov equation and Wasserstein constraints. Specifically, for each output \( z_i \in \{ \omega, h, \phi, u \} \), we solve
\begin{align}
	\maximize_{\dvrMax{\Sigma} \in \scrS_\text{cor}} &\quad \tr\left( \begin{pmatrix} \dvrC{\calC_i} & \dvrC{\calD_i} & 0 \end{pmatrix} \dvrMax{\Sigma} \begin{pmatrix} \dvrC{\calC_i} & \dvrC{\calD_i} & 0 \end{pmatrix}^\top \right), \label{eq:expMomentProblem} \\
	\mathrm{s.t.} &\quad \eqref{eq:momentRelaxedProblemB} - \eqref{eq:momentRelaxedProblemD}, \notag
\end{align}
where \( \dvrC{\calC_i} \) and \( \dvrC{\calD_i} \) are the rows of the closed-loop matrices corresponding to the output \( z_i \). This approach targets each output with a worst-case correlated wind disturbance, as the ambiguity set \( \scrW_\text{cor} \) allows disturbances to exploit temporal dependencies.

\begin{table}[t]
	\centering
	\caption{Worst-Case Output Variances of $H_2$ and DR Controllers synthesized for unscaled outputs.}
	\label{tab:results}
	\begin{tabular}{lcc}
		\toprule
		\textbf{Output (Unit)} & \textbf{$H_2$ Variance} & \textbf{DR Variance} \\
		\midrule
		Rotor speed (rad$^2$/s$^2$) & 25.781 & 0.388 \\
		Flapwise disp. (m$^2$) & 22.246 & 20.623 \\
		Torsional angle (rad$^2$) & 0.000385 & 0.000383 \\
		Control input (rad$^2$/s$^2$) & 8.453 & 11.387 \\
		\bottomrule
	\end{tabular}
\end{table}

\begin{table}[t]
	\centering
	\caption{Worst-Case Output Variances of $H_2$ and DR Controllers synthesized for scaled outputs.}
	\label{tab:resultsForScaledOutputs}
	\begin{tabular}{lcc}
		\toprule
		\textbf{Output (Unit)} & \textbf{$H_2$ Variance} & \textbf{DR Variance} \\
		\midrule
		Rotor speed (rad$^2$/s$^2$) & 16.889 & 0.129 \\
		Flapwise disp. (m$^2$) & 10.424 & 3.906 \\
		Torsional angle (rad$^2$) & 0.000385 & 0.000383 \\
		Control input (rad$^2$/s$^2$) & 45.670 & 142.453 \\
		\bottomrule
	\end{tabular}
\end{table}

Table~\ref{tab:results} reports the worst-case variances for each performance output under individual worst-case disturbance attacks. Clearly, both controllers do not yet meet the design requirements for flapwise displacement, since the target RMS is 2 m, corresponding to a variance of 4 m$^2$. For this reason, we carry out a second design step, where we scale the outputs statically to meet the design requirements. The scaled outputs are defined as \( \tilde{z} = (\tilde{z}_1 ~ \tilde{z}_2~ \tilde{z}_3~ \tilde{z}_4)^\top = (\omega, 5 h, 100 \phi, u)^\top \), where the scaling of \( h \) is chosen to meet the design requirement of 2 m RMS, and the scaling of \( \phi \) is chosen to account for the dimensions of the torsional angle.

Table~\ref{tab:resultsForScaledOutputs} shows the worst-case variances for the scaled outputs. After scaling, the DR controller achieves significantly lower variances for the rotor speed \( \omega \) and a lower flapwise displacement \( h \) compared to the $H_2$ controller. In particular, the DR controller is now able to meet the design requirements. The torsional angle \( \phi \) remains nearly identical for both controllers, while the control input \( u \) shows a higher variance for the DR controller, indicating that more aggressive control actions are required to mitigate the disturbances in the flapwise and rotor speed outputs.

The results highlight the capabilities of the DR controller to mitigate worst-case disturbances, particularly for \( \omega \) and \(h\), by accounting for distributional uncertainty in \( \scrW_\text{cor} \). This is confirmed by the Bode diagrams supplemented in Figure \ref{fig:expRotorSpeed}, which display the amplitude response from the wind disturbance to each of the performance outputs. We see that the DR controller achieves a significantly lower amplitude response for the rotor speed \( \omega \) compared to the $H_2$ controller. The amplitude responses for the torsional angle \( \phi \) are essentially identical for both controllers, while the control input \( u \) shows a higher response for the DR controller. The peaks of these amplitude responses correspond to the frequencies where the disturbances have the most significant impact on the outputs. We emphasize that we believe that $H_2$ controllers can be designed to meet the design requirements as well (e.g., with frequency dependent weightings). However, for the specific performance criteria investigated here, the DR controller should be easier to tune, since it is designed to meet those criteria directly.

We mention that representative worst-case disturbances designed with \eqref{eq:expMomentProblem} yield wind variances of approximately 1.94 m$^2$/s$^2$ (axial) and 1.72 m$^2$/s$^2$ (vertical). These are realistic for offshore or moderate onshore conditions at wind speeds of \(9.5\) m/s, aligning with typical turbulence models (turbulence intensity \(\approx 14.6\%\)) \cite{burton2011wind}.

\begin{figure}
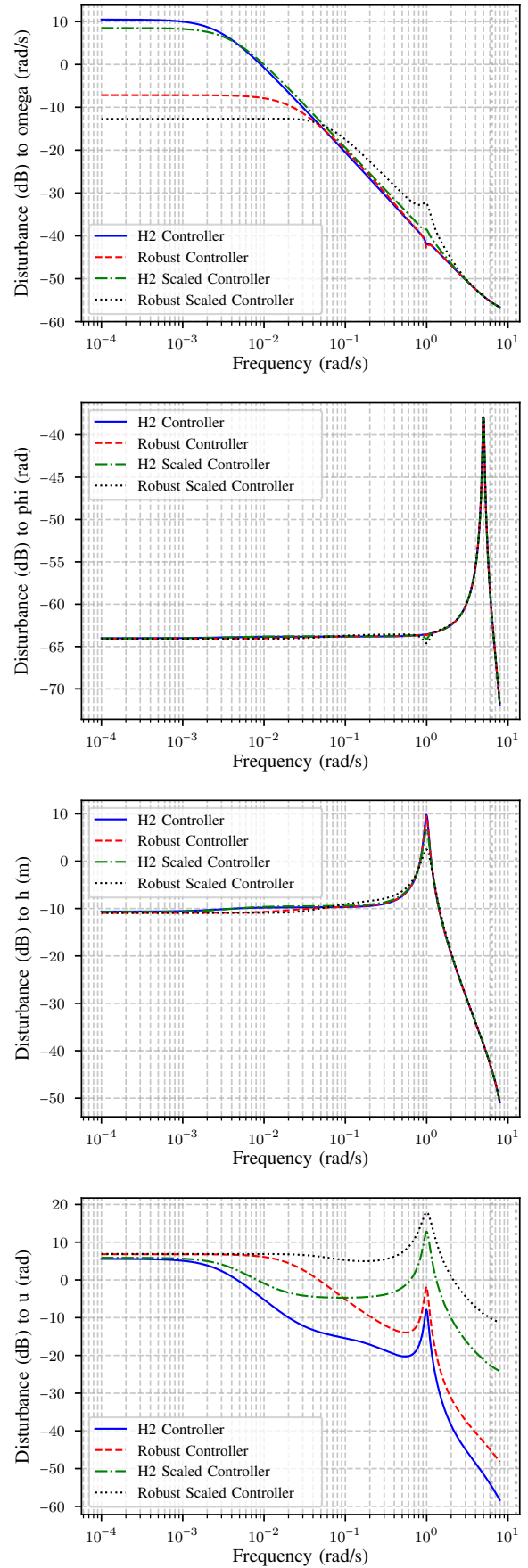

  \begin{subfigure}{0.9\linewidth}
    \centering
    \resizebox{\linewidth}{!}{\input{omega_svd.pgf}}
    \end{subfigure}
    \begin{subfigure}{0.9\linewidth}
    \centering
    \resizebox{\linewidth}{!}{\input{phi_svd.pgf}}
    \end{subfigure}
    \begin{subfigure}{0.9\linewidth}
    \centering
    \resizebox{\linewidth}{!}{\input{h_svd.pgf}}
    \end{subfigure}
    \begin{subfigure}{0.9\linewidth}
    \centering
    \resizebox{\linewidth}{!}{\input{u_svd.pgf}}
    \end{subfigure}
    \caption{Bode diagrams from the maximum singular value of the disturbance to each performance output, $\omega$, $\phi$, $h$, $u$.}
    \label{fig:expRotorSpeed}
\end{figure}

\section{Conclusion}
\label{sec:conclusion}

In this article, we establish a fundamental connection between distributionally robust controller synthesis and robust $H_2$ control theory. Our main contribution is the formulation of DRC problems with Wasserstein ambiguity sets as convex LMI synthesis problems. This reformulation is exact and allows for efficient numerical implementation using standard SDP solvers or, even more efficiently, using structure exploiting algorithms \cite{gramlich2023structure}. Both correlated and independent disturbances are handled within the same mathematical framework, and we provide explicit convex formulations for both cases.

The modular nature of our approach extends beyond pure Wasserstein constraints. As sketched in Appendix~\ref{app:mergingAmbiguitySets}, the framework readily accommodates hybrid ambiguity sets that combine distributional uncertainty with autocorrelation constraints, opening avenues for incorporating diverse forms of prior knowledge about disturbance characteristics. This flexibility suggests promising directions for future research in DRC with structured uncertainty descriptions.

\bibliographystyle{plain}
\bibliography{sources}

\appendix

\subsection{Proof of Theorem~\ref{thm:finiteHorizonProblemEquivalent}}
\label{app:equivalence}

Let $(w, \hat{w})$ be a feasible point of \eqref{eq:worstCaseProblemAlternative}. By \eqref{eq:worstCaseProblemAlternativeC}, the joint distribution $\mathbb{P}_{(w(t), \hat{w}(t))}$ is a coupling in $\Pi(\mathbb{P}_{w(t)}, \mathbb{P}_{\hat{w}(t)})$ such that $\mathbb{W}(\mathbb{P}_{w(t)}, \mathbb{P}_{\hat{w}(t)}) = \mathbb{W}(\mathbb{P}_{w(t)}, \mathbb{P}_{\text{nom}}) \leq \gamma$ per \eqref{eq:worstCaseProblemAlternativeD}. Thus, $w$ satisfies the feasibility constraint of \eqref{eq:finiteHorizonProblem}. Since the objective essentially depends only on the choice of $w$, the values match.

Now, let $w \in \mathscr{W}$ be a feasible point of \eqref{eq:finiteHorizonProblem}. For each $t \in \mathbb{N}_0$, there exists a coupling $\pi_t \in \Pi(\mathbb{P}_{w(t)}, \mathbb{P}_{\text{nom}})$ on $\mathbb{R}^d \times \mathbb{R}^d$ with $\mathbb{W}(\mathbb{P}_{w(t)}, \mathbb{P}_{\text{nom}}) \leq \gamma$ (see \cite{villani2008optimal}, Theorem 4.1). We construct a probability space $(\Omega^\infty, \mathbb{P}^\infty)$ supporting $w$ and $\hat{w}$ such that $\mathbb{P}^\infty_{(w(t), \hat{w}(t))} = \pi_t$ and $\mathbb{P}^\infty_w = \mathbb{P}_w$.

Consider the separable, metric sample space $\Omega$ with measure $\mathbb{P}$ supporting $w$. For each $t \in \mathbb{N}_0$, define:
\[
\Omega^t = \Omega \times \left( \mathbb{R}^d \right)^{t+1},
\]
supporting $(w, \hat{w}(0), \ldots, \hat{w}(t))$, where $\hat{w}(t') : \mathbb{R}^d \to \mathbb{R}^d$ for $t' = 0, \ldots, t$.

Construct measures $\mathbb{P}^t$ on $\Omega^t$ using the gluing lemma \cite{villani2008optimal}. For $t = 0$, set $\Omega^0 = \Omega \times \mathbb{R}^d$. Glue $\mathbb{P}$ on $\Omega$ (inducing $\mathbb{P}_{w(0)}$ on $\mathbb{R}^d$ via $w(0)$) with $\pi_0$ on $\mathbb{R}^d \times \mathbb{R}^d$ (first marginal $\mathbb{P}_{w(0)}$) to obtain $\mathbb{P}^0$ on $\Omega^0$ such that:
\begin{itemize}
  \item The marginal on $\Omega$ is $\mathbb{P}$,
  \item The joint distribution of $(w(0), \hat{w}(0))$ is $\pi_0$.
\end{itemize}

For $t \geq 0$, given $\mathbb{P}^t$ on $\Omega^t$, define $\Omega^{t+1} = \Omega^t \times \mathbb{R}^d = \Omega \times \left( \mathbb{R}^d \right)^{t+2}$. Glue $\mathbb{P}^t$ on $\Omega^t$ (inducing $\mathbb{P}_{w(t+1)}$ on $\mathbb{R}^d$ via $w(t+1)$) with $\pi_{t+1}$ on $\mathbb{R}^d \times \mathbb{R}^d$ (first marginal $\mathbb{P}_{w(t+1)}$) to obtain $\mathbb{P}^{t+1}$ on $\Omega^{t+1}$ such that:
\begin{itemize}
  \item The marginal on $\Omega^t$ is $\mathbb{P}^t$,
  \item The joint distribution of $(w(t+1), \hat{w}(t+1))$ is $\pi_{t+1}$.
\end{itemize}

The measures $\mathbb{P}^t$ are consistent: the projection of $\mathbb{P}^{t+1}$ onto $\Omega^t$ is $\mathbb{P}^t$, as the gluing lemma preserves marginals at each step.

Define the infinite space:
\[
\Omega^\infty = \Omega \times \prod_{t=0}^\infty \mathbb{R}^d,
\]
with the product $\sigma$-algebra, supporting $(\omega, \hat{w}(0), \hat{w}(1), \ldots)$. The projection $p^t : \Omega^\infty \to \Omega^t$ is:
\[
p^t(\omega, \hat{w}(0), \hat{w}(1), \ldots) = (\omega, \hat{w}(0), \ldots, \hat{w}(t)).
\]
Since $\mathbb{P}^{t+1} \circ (p^t)^{-1} = \mathbb{P}^t$, Kolmogorov’s extension theorem yields a unique $\mathbb{P}^\infty$ on $\Omega^\infty$ such that $\mathbb{P}^\infty \circ (p^t)^{-1} = \mathbb{P}^t$ for all $t$.

On $\Omega^\infty$, define:
\begin{align*}
  w(\omega, \hat{w}(0), \hat{w}(1), \ldots) &= w(\omega),\\
  \hat{w}(\omega, \hat{w}(0), \hat{w}(1), \ldots)(t) = \hat{w}(t).
\end{align*}
Under $\mathbb{P}^\infty$, $\mathbb{P}^\infty_w = \mathbb{P}_w$ (marginal on $\Omega$ is $\mathbb{P}$), and $\mathbb{P}^\infty_{(w(t), \hat{w}(t))} = \pi_t$ for each $t$, as ensured by $\mathbb{P}^t$.

Thus, $(w, \hat{w})$ on $(\Omega^\infty, \mathbb{P}^\infty)$ satisfies the constraints of \eqref{eq:worstCaseProblemAlternative}, and the objective values equal those of $w$ in \eqref{eq:finiteHorizonProblem}. This establishes the one-to-one correspondence, completing the proof.
\hfill $\square$

\subsection{Wasserstein and Autocorrelation Conditions}
\label{app:mergingAmbiguitySets}

While the main contribution of this article concerns Wasserstein ambiguity sets, we wish to highlight the advantages of the modular robust control framework by briefly sketching the compatibility of the Wasserstein ambiguity set description presented in this article with autocorrelation conditions as proposed in \cite{gusev1995minimax_a,scherer2000robust}.

In this course, let us assume that the autocorrelation terms of the disturbance input $w \in \scrW_\text{cor}$ satisfy additionally the $N$ constraints
\begin{align}
  \limsup_{T \to \infty} \frac{1}{T-l} \sum_{t = l}^{T-1} 
  \sum_{k=0}^{l} \tr \big( \bbE w_{t-k} w_t^\top  M_i^k \big) \leq \gamma_i, \quad i = 1, \ldots, N \label{eq:autocorrelationConditions}
\end{align}
defined by matrices $M_i^k \in \bbR^{n_w \times n_w}$ and $\gamma_i \in \bbR$. This restriction can be formulated in the (integral quadratic constraint) form
\begin{align*}
  \limsup_{T \to \infty} \frac{1}{T} \sum_{t = 0}^{T-1} \bbE w_t^\top (C^\psi_i x^\psi_t + D_i^\psi w_t) \leq \gamma_i, \quad i = 1, \ldots, N,
\end{align*}
where $x^\psi$ is the stacked vector of delayed disturbance inputs
\begin{align*}
  x_t^\psi = \begin{pmatrix}
    w_{t-1}^\top & w_{t-2}^\top & \cdots & w_{t-l}^\top
  \end{pmatrix}^\top \in \bbR^{(l+1)n_w}.
\end{align*}
and $C^\psi_i$ and $D_i^\psi$ are matrices that captures the terms $M_i^k$ in the sense that
\begin{align*}
  C^\psi_i = \begin{pmatrix}
    M_i^1 & \cdots & M_i^l
  \end{pmatrix}, \quad D_i^\psi = M_i^0.
\end{align*}
We emphasize that $x_t^\psi$ can be recursively defined through an appropriate state-space representation of the form
\begin{align*}
  x_{t+1}^\psi = A^\psi x_t^\psi + B^\psi w_t.
\end{align*}
Next, we incorporate the prior knowledge that the disturbance sequence $w$ satisfies the autocorrelation conditions \eqref{eq:autocorrelationConditions} into the moment problem \eqref{eq:momentRelaxedProblem}. The first step is the definition of the extended moment matrix $\dvrMax{\Sigma}$, which is given by
\begin{align*}
  \begin{pmatrix}
    \dvrMax{\Sigma_{\chi\chi}} & \dvrMax{\Sigma_{\chi x^\psi}} & \dvrMax{\Sigma_{\chi w}} & \dvrMax{\Sigma_{\chi\hat{w}}}\\
    \dvrMax{\Sigma_{x^\psi\chi}} & \dvrMax{\Sigma_{x^\psi x^\psi}} & \dvrMax{\Sigma_{x^\psi w}} & \dvrMax{\Sigma_{x^\psi\hat{w}}}\\
    \dvrMax{\Sigma_{w\chi}} & \dvrMax{\Sigma_{w x^\psi}} & \dvrMax{\Sigma_{ww}} & \dvrMax{\Sigma_{w\hat{w}}}\\
    \dvrMax{\Sigma_{\hat{w}\chi}} & \dvrMax{\Sigma_{\hat{w}x^\psi}} & \dvrMax{\Sigma_{\hat{w}w}} & \dvrMax{\Sigma_{\hat{w}\hat{w}}}
  \end{pmatrix}.
\end{align*}
Subsequently, we define the modified moment problem
\begin{subequations}
  \label{eq:momentRelaxedProblem2}
\begin{align}
  \maximize_{\dvrMax{\Sigma} \in \scrS} ~&~ \tr\left(
  \begin{pmatrix}
    \dvrC{\calC} & 0 & \dvrC{\calD} & 0
  \end{pmatrix}
  \dvrMax{\Sigma}
  \begin{pmatrix}
    \dvrC{\calC} & 0 & \dvrC{\calD} & 0
  \end{pmatrix}^\top
  \right), \label{eq:momentRelaxedProblem2A}
\end{align}
subject to
\begin{align}
  &
  % \dvrMax{\Sigma_{\chi\chi}}
  \begin{pmatrix}
    I & 0 & 0 & 0\\
    0 & I & 0 & 0
  \end{pmatrix}
  \dvrMax{\Sigma}
  \begin{pmatrix}
    I & 0 & 0 & 0\\
    0 & I & 0 & 0
  \end{pmatrix}^\top \nonumber\\
  &\qquad
  =
  \begin{pmatrix}
    \dvrC{\calA} & 0 & \dvrC{\calB} & 0\\
    0 & A^\psi & B^\psi & 0
  \end{pmatrix}
  \dvrMax{\Sigma}
  \begin{pmatrix}
    \dvrC{\calA} & 0 & \dvrC{\calB} & 0\\
    0 & A^\psi & B^\psi & 0
  \end{pmatrix}^\top,\label{eq:momentRelaxedProblem2B}\\
  &
  \gamma^2 \geq
  \tr\left(
  \begin{pmatrix}
    0 & 0 & I & 0\\
    0 & 0 & 0 & I
  \end{pmatrix}
  \dvrMax{\Sigma}
  \begin{pmatrix}
    0 & 0 & I & 0\\
    0 & 0 & 0 & I\\
  \end{pmatrix}^\top
  \begin{pmatrix}
    I & -I\\
    -I & I
  \end{pmatrix}
  \right) , \label{eq:momentRelaxedProblem2C}\\
  & \Sigma_\text{nom} = \begin{pmatrix}
    0 & 0 & 0 & I
  \end{pmatrix}
  \dvrMax{\Sigma}
  \begin{pmatrix}
    0 & 0 & 0 & I
  \end{pmatrix}^\top, \label{eq:momentRelaxedProblem2E}\\
  & \dvrMax{\Sigma} \succeq 0, \label{eq:momentRelaxedProblem2D}\\
  &
  \tr \left(
    \begin{pmatrix}
      0 & C_i^\psi & D_i^\psi & 0\\
      0 & 0 & I & 0
    \end{pmatrix}
    \dvrMax{\Sigma}
    \begin{pmatrix}
      0 & 0 & I & 0\\
      0 & C_i^\psi & D_i^\psi & 0
    \end{pmatrix}^\top
  \right) \leq 2 \gamma_i,\nonumber\\
  &\qquad i = 1, \ldots, N. \label{eq:momentRelaxedProblem2F}
\end{align}
\end{subequations}
Now, with the modified moment problem \eqref{eq:momentRelaxedProblem2}, we can carry out the very same steps as in Section \ref{subsec:robustPerformanceAnalysisCor}. Namely, we can carry out Lagrange dualization of \eqref{eq:momentRelaxedProblem2} to obtain the dual problem
\begin{subequations}
  \label{eq:performanceAnalysisCor2}
\begin{align}
    &\minimize_{\dvr{\lambda} \geq 0, \dvr{\mu} \geq 0,\dvr{\calP}, \dvr{Q}} ~~ \tr\left( \dvr{Q} \Sigma_\text{nom} \right) + \dvr{\lambda}\gamma^2 + 2 \sum_{i = 0}^N \gamma_i \dvr{\mu_i} \label{eq:performanceAnalysisCor2A}
\end{align}
subject to
\begin{align}
  \resizebox{\linewidth}{!}{$
    \begin{pmatrix}
      \bullet
    \end{pmatrix}^\top
    \begin{pmatrix}
      -\dvr{\calP_{11}} & -\dvr{\calP_{12}} & 0 & 0 & 0 & 0 & 0\\
      -\dvr{\calP_{21}} & -\dvr{\calP_{22}} & 0 & 0 & 0 & 0 & 0\\
      0 & 0 & \dvr{\calP_{11}} & \dvr{\calP_{12}} & 0 & 0 & 0\\
      0 & 0 & \dvr{\calP_{21}} & \dvr{\calP_{22}} & 0 & 0 & 0\\
      0 & 0 & 0 & I & 0 & 0 & 0\\
      0 & 0 & 0 & 0 & 0 & I & 0\\
      0 & 0 & 0 & 0 & I & -\dvr{\lambda}I & \dvr{\lambda}I\\
      0 & 0 & 0 & 0 & 0 & \dvr{\lambda}I & -\dvr{Q} - \dvr{\lambda}I
    \end{pmatrix}
    \begin{pmatrix}
      I & 0 & 0 & 0\\
      0 & I & 0 & 0\\
      \dvrC{\calA} & 0 & \dvrC{\calB} & 0\\
      0 & A^\psi & B^\psi & 0\\
      \dvrC{\calC} & 0 & \dvrC{\calD} & 0\\
      0 & C^\psi(\dvr{\mu}) & D^\psi(\dvr{\mu}) & 0\\
      0 & 0 & I & 0\\
      0 & 0 & 0 & I
    \end{pmatrix} \prec 0.$} \label{eq:performanceAnalysisCor2B}
  \end{align}
\end{subequations}
Here, $C^\psi(\dvr{\mu}) := \sum_{i=1}^N \dvr{\mu_i} C^\psi_i$ and $D^\psi(\dvr{\mu}) := \sum_{i=1}^N \dvr{\mu_i} D^\psi_i$ are the matrices that capture the autocorrelation conditions \eqref{eq:autocorrelationConditions} and $\dvr{\mu_i}$ are the Lagrange multipliers of these constraints. Further, the matrices $A^\psi$, $B^\psi$, $C^\psi(\dvr{\mu})$, and $D^\psi(\dvr{\mu})$ are exactly the filter matrices studied in \cite{scherer2000robust} to capture the autocorrelation conditions.

We emphasize that the question of weak and strong duality needs to be reconsidered in the derivation of \eqref{eq:performanceAnalysisCor2}.

The analysis problem \eqref{eq:performanceAnalysisCor2} is similar to the one presented in \cite{scherer2000robust}. Thus, synthesis can be convexified as in \cite{scherer2000robust} using the \emph{elimination lemma}.
\end{document}

%% file: stateSpaceDeltaStructureController.tex
\tikzstyle{block} = [draw,rectangle,thick,minimum height=2em,minimum width=2em,
shade, shading=axis, top color=lightgray, bottom color=white,
rounded corners=.15cm]
\tikzstyle{sum} = [draw,circle,inner sep=0mm,minimum size=0mm]
\tikzstyle{connector} = [-,thick]
\tikzstyle{arrowconnector} = [-stealth,thick]

\begin{tikzpicture}[scale=1, auto]
		\node (e1) {};
		\node[block,right = 2cm of e1] (f1) {$Z_-$};
		\node[right = 2cm of f1] (e2) {};

		\node[block, below = 0.5cm  of f1] (f2) {$\begin{pmatrix}
				x^+\\
				\hline
				z\\
			\end{pmatrix}
			=
			\left(
				\begin{array}{c|ccc}
					\dvrC{\calA} & \dvrC{\calB}\\
					\hline
					\dvrC{\calC} & \dvrC{\calD}
				\end{array}
			\right)
			\begin{pmatrix}
				x\\ 
				\hline 
				w
			\end{pmatrix}$};
		\node [left = 0.5cm of f2] (e3) {}; 
		\node [right = 0.5cm of f2] (e4) {};
		
		\node[block,below left = -0.8cm and 1.0cm of f2] (f4) {$\Delta$};
		\node [left = 1cm of f4] (e7) {}; 
	
	% now link the nodes
	\draw [connector] (e1.center) -- node{$x$} (f1);
	\draw [arrowconnector] (e2.center) -- node{$x^+$} (f1);
	\draw [arrowconnector] (e1.center) |- ($(f2.west) + (0,0.2)$);
	\draw [connector] (e2.center) |- node{$~$} ($(f2.east) + (0,0.2)$);
	\draw [arrowconnector] ($(f4.east) - (0.0,0)$) -- node{$w$} ($(f2.west) - (0,0.2)$);
	\draw [arrowconnector] ($(f2.east) - (0,0.2)$) -- node{$z$} ($(e4.east) - (-0.5,0.2)$);
	\draw[arrowconnector] ($(e7.center) - (0,0.0)$) -- node{$\hat{w}$} ($(f4.west) - (0,0.0)$);
\end{tikzpicture}

%% file: main.bbl
\begin{thebibliography}{10}

\bibitem{brouillon2023distributionally}
Jean-S{\'e}bastien Brouillon, Andrea Martin, John Lygeros, Florian D{\"o}rfler, and Giancarlo Ferrari-Trecate.
\newblock Distributionally robust infinite-horizon control: from a pool of samples to the design of dependable controllers.
\newblock {\em IEEE Transactions on Automatic Control}, pages 1--16, 2025.

\bibitem{burton2011wind}
Tony Burton, Nick Jenkins, David Sharpe, and Ervin Bossanyi.
\newblock {\em Wind energy handbook}.
\newblock John Wiley \& Sons, 2011.

\bibitem{delage2010distributionally}
Erick Delage and Yinyu Ye.
\newblock Distributionally robust optimization under moment uncertainty with application to data-driven problems.
\newblock {\em Operations research}, 58(3):595--612, 2010.

\bibitem{gelbrich1990formula}
Matthias Gelbrich.
\newblock On a formula for the {$L_2$} {Wasserstein} metric between measures on {Euclidean} and {Hilbert} spaces.
\newblock {\em Mathematische Nachrichten}, 147(1):185--203, 1990.

\bibitem{gramlich2024moment}
Dennis Gramlich, Sheng Gao, Hao Zhang, Carsten~W Scherer, and Christian Ebenbauer.
\newblock On moment relaxations for linear state feedback controller synthesis with non-convex quadratic costs and constraints.
\newblock {\em arXiv preprint arXiv:2403.15228}, 2024.

\bibitem{gramlich2023structure}
Dennis Gramlich, Tobias Holicki, Carsten~W Scherer, and Christian Ebenbauer.
\newblock A structure exploiting {SDP} solver for robust controller synthesis.
\newblock {\em IEEE Control Systems Letters}, 7:1831--1836, 2023.

\bibitem{gusev1995minimax_a}
Sergei~V Gusev.
\newblock Minimax control under a bound on the partial covariance sequence of the disturbance.
\newblock {\em Automatica}, 31(9):1287--1301, 1995.

\bibitem{gusev1995minimax_b}
Sergei~V Gusev.
\newblock Minimax control under a restriction on the moments of disturbance.
\newblock In {\em Proceedings of 1995 34th IEEE Conference on Decision and Control}, volume~2, pages 1195--1200. IEEE, 1995.

\bibitem{gusev1996method}
Sergei~V Gusev.
\newblock Method of moment restrictions in robust control and filtering.
\newblock {\em IFAC Proceedings Volumes}, 29(1):3805--3810, 1996.

\bibitem{10313386}
Joudi Hajar, Taylan Kargin, and Babak Hassibi.
\newblock Wasserstein {D}istributionally {R}obust {R}egret-{O}ptimal {C}ontrol under {P}artial {O}bservability.
\newblock In {\em 2023 59th Annual Allerton Conference on Communication, Control, and Computing (Allerton)}, pages 1--6, 2023.

\bibitem{9992738}
Astghik Hakobyan and Insoon Yang.
\newblock Wasserstein {D}istributionally {R}obust {C}ontrol of {P}artially {O}bservable {L}inear {S}ystems: {T}ractable {A}pproximation and {P}erformance {G}uarantee.
\newblock In {\em 2022 IEEE 61st Conference on Decision and Control}, pages 4800--4807, 2022.

\bibitem{han2024distributionally}
Bingyan Han.
\newblock Distributionally robust {K}alman filtering with volatility uncertainty.
\newblock {\em IEEE Transactions on Automatic Control}, 2024.

\bibitem{kargin2024infinite}
Taylan Kargin, Joudi Hajar, Vikrant Malik, and Babak Hassibi.
\newblock Infinite-{H}orizon {D}istributionally {R}obust {R}egret-{O}ptimal {C}ontrol.
\newblock In {\em Proceedings of the 41st International Conference on Machine Learning}, pages 23187--23214, 2024.

\bibitem{kim2023distributional}
Kihyun Kim and Insoon Yang.
\newblock Distributional robustness in minimax linear quadratic control with {W}asserstein distance.
\newblock {\em SIAM Journal on Control and Optimization}, 61(2):458--483, 2023.

\bibitem{kotsalis2021convex}
Georgios Kotsalis, Guanghui Lan, and Arkadi~S Nemirovski.
\newblock Convex optimization for finite-horizon robust covariance control of linear stochastic systems.
\newblock {\em SIAM Journal on Control and Optimization}, 59(1):296--319, 2021.

\bibitem{kuhn2019wasserstein}
Daniel Kuhn, Peyman~Mohajerin Esfahani, Viet~Anh Nguyen, and Soroosh Shafieezadeh-Abadeh.
\newblock Wasserstein distributionally robust optimization: {T}heory and applications in machine learning.
\newblock In {\em Operations research \& management science in the age of analytics}, pages 130--166. Informs, 2019.

\bibitem{luenberger1997optimization}
David~G Luenberger.
\newblock {\em Optimization by vector space methods}.
\newblock John Wiley \& Sons, 1997.

\bibitem{masubuchi1998lmi}
Izumi Masubuchi, Atsumi Ohara, and Nobuhide Suda.
\newblock {LMI-based} controller synthesis: a unified formulation and solution.
\newblock {\em International Journal of Robust and Nonlinear Control: IFAC-Affiliated Journal}, 8(8):669--686, 1998.

\bibitem{mohajerin2018data}
Peyman Mohajerin~Esfahani and Daniel Kuhn.
\newblock Data-driven distributionally robust optimization using the {W}asserstein metric: Performance guarantees and tractable reformulations.
\newblock {\em Mathematical Programming}, 171(1):115--166, 2018.

\bibitem{nguyen2023bridging}
Viet~Anh Nguyen, Soroosh Shafieezadeh-Abadeh, Daniel Kuhn, and Peyman Mohajerin~Esfahani.
\newblock Bridging bayesian and minimax mean square error estimation via {W}asserstein distributionally robust optimization.
\newblock {\em Mathematics of Operations Research}, 48(1):1--37, 2023.

\bibitem{599969}
C.~Scherer, P.~Gahinet, and M.~Chilali.
\newblock Multiobjective output-feedback control via {LMI} optimization.
\newblock {\em IEEE Transactions on Automatic Control}, 42(7):896--911, 1997.

\bibitem{scherer2000robust}
Carsten Scherer.
\newblock Robust controller design by output feedback against uncertain stochastic disturbances.
\newblock {\em IFAC Proceedings Volumes}, 33(14):453--458, 2000.

\bibitem{scherer2000linear}
Carsten Scherer and Siep Weiland.
\newblock Linear matrix inequalities in control.
\newblock {\em Lecture Notes, Dutch Institute for Systems and Control, Delft, The Netherlands}, 3(2), 2000.

\bibitem{10384311}
Feras~Al Taha, Shuhao Yan, and Eilyan Bitar.
\newblock A distributionally robust approach to regret optimal control using the {Wasserstein} distance.
\newblock In {\em 2023 62nd IEEE Conference on Decision and Control}, pages 2768--2775, 2023.

\bibitem{taskesen2024distributionally}
Bahar Taskesen, Dan Iancu, {\c{C}}a{\u{g}}{\i}l Ko{\c{c}}yi{\u{g}}it, and Daniel Kuhn.
\newblock Distributionally robust linear quadratic control.
\newblock In {\em Advances in Neural Information Processing Systems}, volume~36, pages 18613--18632. Curran Associates, Inc., 2023.

\bibitem{van2015distributionally}
Bart~PG Van~Parys, Daniel Kuhn, Paul~J Goulart, and Manfred Morari.
\newblock Distributionally robust control of constrained stochastic systems.
\newblock {\em IEEE Transactions on Automatic Control}, 61(2):430--442, 2015.

\bibitem{villani2008optimal}
C{\'e}dric Villani et~al.
\newblock {\em Optimal transport: old and new}, volume 338.
\newblock Springer, 2008.

\bibitem{YBS25_rocond}
Shuhao Yan and Carsten~W. Scherer.
\newblock Distributional robustness in output feedback regret-optimal control.
\newblock In {\em 11th IFAC Symposium on Robust Control Design}, 2025.

\end{thebibliography}
